\documentclass[12pt]{article}   
   
\usepackage{latexsym, amsfonts, amsmath, epsfig, tabularx, amsthm, amssymb, enumitem, bm}
\usepackage{float}

\usepackage{microtype}

\usepackage[utf8]{inputenc}
\usepackage[T1]{fontenc}
\usepackage{lmodern}
\usepackage{extarrows}
\usepackage{esint}
\usepackage{aliascnt}
\usepackage{mathtools}
\usepackage{graphicx}

\usepackage[draft=false,colorlinks,bookmarksnumbered,linkcolor=black, citecolor=black]{hyperref}

\usepackage{aliascnt}
\newtheorem{theorem}{Theorem}[section]

\newaliascnt{lem}{theorem}
\newtheorem{lemma}[lem]{Lemma}

\aliascntresetthe{lem}

\newaliascnt{ass}{theorem}

\aliascntresetthe{ass}

\newaliascnt{prop}{theorem}
\newtheorem{Prop}[prop]{Proposition}

\aliascntresetthe{prop}

\newaliascnt{cor}{theorem}

\aliascntresetthe{cor}

\newaliascnt{defi}{theorem}
\newtheorem{defi}[defi]{Definition}

\aliascntresetthe{defi}

\theoremstyle{definition}
\newaliascnt{ex}{theorem}

\aliascntresetthe{ex}

\newaliascnt{rem}{theorem}
\newtheorem{Rem}[rem]{Remark}

\aliascntresetthe{rem}

\newcommand*{\R}{\mathbb{R}}                                            
\newcommand*{\N}{\mathbb{N}}                                            
\newcommand*{\Z}{\mathbb{Z}}                                            
\newcommand*{\J}{\mathcal{J}}                                         	
\newcommand*{\I}{\mathcal{I}}                                         	
\newcommand*{\E}{\mathcal{E}}                                         	
\newcommand*{\D}{\mathcal{D}}                                         	
\newcommand*{\K}{\mathcal{K}} 	
\renewcommand*{\P}{\mathcal{P}} 	
\renewcommand{\d}{\,\mathrm{d}}											
\newcommand*{\abs}[1]{\left| #1 \right|}                                
\newcommand*{\norm}[1]{\left\| #1 \right\|}                             
\newcommand*{\distr}[2]{\left\langle #1, #2 \right\rangle}              
\newcommand*{\sep}{\; \vrule \;}                                        
\newcommand{\esssup}{ \mathop{\text{ess-sup}}\limits }

\newcommand*{\floor}[1]{\left\lfloor #1 \right\rfloor}                  
\newcommand*{\link}[1]{(\ref{#1})}                                      
\newcommand*{\loc}{\mathrm{loc}}
\newcommand*{\dist}{\mathrm{dist}}
\newcommand*{\diam}{\mathrm{diam}}
\renewcommand{\tilde}{\widetilde}

\usepackage{graphicx}
\usepackage{epstopdf}
\DeclareMathOperator{\supp}{supp}

\renewcommand{\div}[1]{\mathrm{div}\!\left(#1\right)}

\usepackage[format=hang]{caption}

\title{Besov regularity of solutions to the $p$-Poisson equation
\footnote{This work has been supported by Deutsche Forschungsgemeinschaft DFG (DA 360/18-1, DA 360/19-1)
and European Research Council ERC (Starting Grant HDSP-CONTR-306274).}}

\author{Stephan Dahlke \and Lars Diening \and Christoph Hartmann\footnote{Corresponding author.} \and Benjamin Scharf \and Markus Weimar}
\date{\today}

\begin{document}

\maketitle

\begin{abstract}
\noindent In this paper, we study  the regularity of solutions to the $p$-Poisson equation for all $1<p<\infty$.  
In particular, we are interested in  smoothness estimates in the adaptivity scale $ B^\sigma_{\tau}(L_{\tau}(\Omega))$, $1/\tau = \sigma/d+1/p$, of Besov spaces.  
The regularity in this scale determines the order of approximation that can be achieved by adaptive and other nonlinear approximation methods.
It turns out that, especially for solutions to $p$-Poisson equations with homogeneous Dirichlet boundary conditions on bounded polygonal domains, the Besov regularity is significantly higher than the Sobolev regularity which justifies the use of adaptive algorithms. 
This type of results is obtained by combining local H\"older with global Sobolev estimates. 
In particular, we prove that intersections of locally weighted H\"older spaces and Sobolev spaces can be continuously embedded into the specific scale of Besov spaces we are interested in. 
The proof of this embedding result is based on wavelet characterizations of Besov spaces.


\end{abstract}


\section{Introduction}
This paper is concerned with regularity estimates of the solutions to 
the \emph{$p$-Poisson equation}
\begin{align}
-\div{\lvert \nabla u \rvert^{p-2} \nabla u} = f \qquad \text{in} \quad \Omega, \label{eq:p-Poisson}
\end{align}
where $ 1 < p < \infty$ and $\Omega \subset \mathbb{R}^d$ denotes some bounded Lipschitz domain. The corresponding variational formulation is given by
\begin{align}
	\int_{\Omega} \distr{\abs{\nabla u }^{p-2} \nabla u}{\nabla v} \d x 
	= \int_{\Omega} f \,v \d x 
	\qquad \text{for all} \quad v \in C_0^{\infty}(\Omega).
	\label{eq:p-Poisson_variational}
\end{align}

Problems of this type arise in many applications, e.g., in non-Newtonian 
fluid theory, non-Newtonian filtering, turbulent flows of a gas in 
porous media, rheology, radiation of heat and many others. Moreover, 
the $p$-Laplacian has a similar model character for nonlinear problems as 
the ordinary Laplace equation for linear problems.
We refer to \cite{Lin2006} for an introduction. By now, many results 
concerning existence and uniqueness of solution are known, we refer again to \cite{Lin2006} and the references therein. However, in many cases, the 
concrete shape of the solutions is unknown, so that efficient numerical 
schemes for the constructive approximation are needed.  
In practice, e.g., for problems in three and more space dimensions, 
this might lead to systems with hundreds of thousands or even millions of unknown. 
Therefore, a quite natural idea would be to use {\em adaptive} 
strategies to increase efficiency. Essentially, an adaptive algorithm 
is an updating strategy where additional degrees of freedom are only 
spent in regions where the numerical approximation is still ``far away''  
from the exact solution. Nevertheless, although the idea of adaptivity is quite 
convincing, these schemes are hard to analyze and to implement, so 
that some theoretical foundations that justify the use of adaptive 
strategies are highly desirable.

The analysis in this paper is motivated by this problem, in particular 
in connection with adaptive wavelet algorithms. In the wavelet case, 
there is a natural benchmark scheme for adaptivity, and that is best 
$n$-term wavelet approximation. In best $n$-term approximation, one does 
not approximate by linear spaces but by nonlinear manifolds ${\mathcal 
M}_n$, consisting of functions of the form
\begin{equation} \label{nonlinearmanifold}
	S=\sum_{\lambda \in \Lambda} c_{\lambda} \psi_{\lambda},
\end{equation} 
where $\{\psi_{\lambda} \sep \lambda \in {\J}\}$ denotes a given 
wavelet basis and $\Lambda\subset\J$ with $\#\Lambda=n$. We refer to \autoref{sec:spaces} and to the textbooks \cite{Dau1992, 
Mey1992, Woj1997} for further information concerning the construction and 
the basic properties of wavelets. In the wavelet setting, a best 
$n$-term approximation can be realized by extracting the $n$ biggest 
wavelet coefficients from the wavelet expansion of the (unknown) function one wants to approximate.  
Clearly, on the one hand, such a scheme can never be realized numerically, because this would 
require to compute all wavelet coefficients and to select the $n$ 
biggest. On the other hand, the best we can expect for an adaptive 
wavelet algorithm would be that it (asymptotically) realizes the 
approximation order of the best $n$-term approximation. In this sense, 
the use of adaptive schemes is justified if best $n$-term wavelet 
approximation realizes a significantly higher convergence order when 
compared to more conventional, uniform approximation schemes. In the 
wavelet setting, it is known that the convergence order of uniform 
schemes with respect to $L_p$  depends on the regularity of the object 
one wants to approximate in the scale $W^s(L_p(\Omega))$ of 
$L_p$-Sobolev spaces, whereas the order of best $n$-term wavelet 
approximation in $L_p$ depends on the regularity in the {\em 
adaptivity scale} $B^{\sigma}_{\tau}(L_{\tau}(\Omega)), 1/\tau=\sigma/d + 1/p$, of 
Besov spaces. We refer to \cite{DahDahDev1997, Dev1998, Han2012} for 
further information. Therefore, the use of adaptive (wavelet) 
algorithms for (\ref{eq:p-Poisson}) would be justified if the Besov smoothness $\sigma$ of the solution in the adaptivity scale of Besov spaces is higher than its 
Sobolev regularity $s$.

For linear second order elliptic equations, a lot of positive results in this direction already exist; see, e.g., \cite{Dah1999c, DahDev1997, DahSic2009}.
In contrast, it seems that not too much is known for nonlinear equations.
The only contribution we are aware of is the paper \cite{DahSic2013} which is concerned with semilinear equations.
In the present paper, we show a first positive result for quasilinear elliptic equations, i.e., for the $p$-Poisson equation \link{eq:p-Poisson}. 
Results of Savar{\'e} \cite{Sav1998} indicate that, on general Lipschitz domains, the 
Sobolev smoothness of the solutions to (\ref{eq:p-Poisson}) is given by  
$s^*=1 + 1/p$ if $2\leq p < \infty$, and by $s^*=3/2$ if $1 < p < 2$.
However, under certain conditions, the solutions 
possess higher regularity away from the boundary, in the sense that they are locally H\"older 
continuous; see, e.g., \cite{DiB1983,Eva1982,Tol1984,Uhl1977,Ura1968}.
The local H\"older semi-norms may explode as one approaches the 
boundary, but this singular behaviour can be controlled by some power of 
the distance to the boundary as shown, e.g., in \cite{DieKapSch2012,KuuMin2014,Lew1983,LinLin2013}. We refer to Section 4 for a detailed exposition.
 (Properties like this very often hold in 
the context of elliptic boundary problems on nonsmooth domains, we refer, 
e.g., to \cite{MazRos2003} and the references therein for details). It turns out that   the 
combination of the global Sobolev smoothness and the local H\"older regularity  
can be used to establish Besov smoothness for the solutions to~\link{eq:p-Poisson}. In many cases, the Besov 
smoothness $\sigma$ is much higher than the Sobolev smoothness $s^*=1+1/p$ or $s^*=3/2$ respectively, so that the use of adaptive schemes is completely justified.

We state our findings in two steps. First of all, we prove a general embedding theorem 
which says that the intersection of a classical Sobolev space and a H\"older space with the properties outlined above
can be embedded into Besov spaces in the adaptivity scale $1/\tau=\sigma/d + 1/p$. It turns out that for a large range of parameters, the Besov smoothness 
is significantly higher compared to the Sobolev smoothness. The proof of this embedding theorem is performed by exploiting the characterizations of
Besov spaces by means of wavelet expansion coefficients. 
Then we verify that under certain natural conditions the solutions to (\ref{eq:p-Poisson}) indeed  satisfy the
assumptions of the embedding theorem, so that its application yields the desired result. 

This paper is organized as follows: In \autoref{sec:spaces}, we introduce all the function spaces that will be used in the paper, including their wavelet characterizations, if possible.
Afterwards, in \autoref{sec:embedding} and \autoref{sec:Besov_reg}, we state and prove our main results: Our general embedding (\autoref{theorem:embedding}) can be found in \autoref{sec:embedding}. 
Its application to the case of the solutions to (\ref{eq:p-Poisson}) which yields new, generic  Besov regularity results (see \autoref{thm:generic_besov} and  \autoref{thm:besov_reg_2d_lip}) is performed in \autoref{subsec:p-Poisson} and \ref{subsec:p-Poisson_d=2}, respectively.
Moreover, here we give explicit bounds on the Besov regularity of the unique solution to the $p$-Poisson equation with homogeneous Dirichlet boundary conditions in two dimensions; see \autoref{thm:besov_reg_2d_lip_max} and \autoref{theorem:Besov_reg_2d}.
The paper is concluded with an Appendix (\autoref{sec:App}) which contains a couple of auxiliary lemmata and propositions which are needed in our proofs. \\  

\noindent\textbf{Notation:} For families $\{a_{\J}\}_{\J}$ and $\{b_{\J}\}_{\J}$ of non-negative real numbers over a common index set we write $a_{\J} \lesssim b_{\J}$ if there exists a constant $c>0$ (independent of the context-dependent parameters $\J$) such that
\begin{equation*}
	a_{\J} \leq c\cdot b_{\J}
\end{equation*}
holds uniformly in $\J$.
Consequently, $a_{\J} \sim b_{\J}$ means $a_{\J} \lesssim b_{\J}$ and $b_{\J} \lesssim a_{\J}$.


\section{Function spaces and wavelet decompositions} \label{sec:spaces}
In this section we recall the definitions of several types of function spaces that will be needed in the sequel. Moreover, we collect some well-known assertions such as, e.g., the characterization of Besov spaces in terms of wavelet coefficients.

\subsection{Strongly differentiable functions: (weighted) H\"older spaces}\label{subsec:Hoelder}
Let $\Omega\subset\R^d$ be some bounded domain, i.e., an open and connected set.
Then, for $\ell\in\N_0$, $C^\ell(\Omega)$ furnished with the norm
\begin{equation*}
	\norm{g \sep C^\ell(\Omega)}
	= \sum_{\abs{\nu}\leq \ell} \sup_{x\in\Omega} \abs{\partial^{\nu}g(x)}
\end{equation*}
denotes the space of all real-valued functions $g$ on $\Omega$ such that $\partial^{\nu}g$ is uniformly continuous and bounded on $\Omega$ for every multi-index $\nu=(\nu_1,\ldots,\nu_d)\in\N_0^d$ with $0\leq \abs{\nu}\leq \ell$. Therein $\partial^\nu=\partial^{\abs{\nu}}/(\partial x_1^{\nu_1}\ldots\partial x_d^{\nu_d})$ denote the $\nu$-th order strong derivatives. 
If $K$ is a compact subset of~$\Omega$ (denoted by $K\subset\subset\Omega$), the spaces $C^\ell(K)$ are defined likewise.
Unless otherwise stated we restrict ourselves to those $K\subset\subset\Omega$ which can be described as the closure of some open and simply connected set.
Next let us recall that for $g\in C^\ell(\Omega)$ the $\ell$-th order H\"older semi-norm with exponent $0<\alpha\leq 1$ is given by
\begin{equation}\label{def:hoelder_seminorm}
	\abs{g}_{C^{\ell,\alpha}(\Omega)} 
	= \sum_{\abs{\nu}=\ell} \sup_{\substack{x,y\in \Omega,\\x\neq y}} \frac{\abs{\partial^{\nu}g(x)-\partial^{\nu}g(y)}}{\abs{x-y}^\alpha}.
\end{equation}
Consequently, for $\ell\in\N_0$ and $0<\alpha\leq 1$,
\begin{equation*}
	C^{\ell,\alpha}(\Omega) = \left\{ g \in C^\ell(\Omega) \sep \norm{g \sep C^{\ell,\alpha}(\Omega)} = \norm{g \sep C^\ell(\Omega)} + \abs{g}_{C^{\ell,\alpha}(\Omega)} < \infty \right\},
\end{equation*}
denote the (classical) \emph{H\"older spaces} on $\Omega$. 
Again we can replace $\Omega$ by $K$ at every occurrence to define the H\"older spaces also for compact subsets $K\subset\subset \Omega$.
Standard proofs yield that all the spaces we defined so far are actually Banach spaces; see, e.g., \cite{Dob2010, Kry1996}. 

Furthermore, let us introduce the collection of 
all functions on $\Omega$ which are  locally H\"older continuous (of order $\ell\in\N_0$ with exponent $0<\alpha\leq 1$). This set will be denoted by 
\begin{align*}
	C^{\ell,\alpha}_{\loc}(\Omega) 
	= \left\{g \colon \Omega \rightarrow \R \sep  g \in C^{\ell,\alpha}(K) \text{ for all } K \subset\subset \Omega \right\},
\end{align*}
where we simplified the notation by denoting the restrictions $g\big|_{K}$ of functions $g$ from $\Omega$ to compact subsets $K$ by $g$ again.
Since the latter collection of functions does not perfectly fit for our purposes, in the sequel the following closely related (non-standard) function spaces will be used instead.
Let $\K$ denote an arbitrary but non-trivial family of compact subsets $K\subset\subset\Omega$.
Then for every $K\in\K$ the quantity
\begin{equation}\label{def:delta}
	\delta_K 
	= \dist ( K, \partial\Omega ),
\end{equation}
i.e., the distance of $K$ to the boundary of $\Omega$, is strictly positive.
Thus, for each $\ell\in\N_0$, all $0<\alpha\leq 1$, and every $\gamma>0$, the space
\begin{align*}
	&C^{\ell,\alpha}_{\gamma,\loc}(\Omega;\K) \\
	&\qquad = \left\{ g \colon \Omega \rightarrow \R \sep g\in C^{\ell,\alpha}(K) \; \text{ for all } \; K\in\K \; \text{ and } \; \abs{g}_{C^{\ell,\alpha}_{\gamma,\loc}} = \sup_{K\in\K} \delta_K^\gamma \abs{g}_{C^{\ell,\alpha}(K)} < \infty \right\}
\end{align*}
is well-defined and it is easily verified that $\abs{\cdot}_{C^{\ell,\alpha}_{\gamma,\loc}}$ provides a semi-norm for this space. 
In our applications below $\K(c)$ will be the set of all closed balls $B=B_r(x_0)\subset\Omega$ (with center $x_0\in\Omega$ and radius $r>0$) such that the (open) ball $\mathring{B}_{c\,r}=\mathring{B}_{c\,r}(x_0)$ is still contained in $\Omega$. Here $c>1$ denotes a constant which we assume to be given fixed in advance. 
Actually, it is not hard to see that the space $C^{\ell,\alpha}_{\gamma,\loc}(\Omega;\K(c))$ is independent of $c$. 
Consequently, we simply write $C^{\ell,\alpha}_{\gamma,\loc}(\Omega)=C^{\ell,\alpha}_{\gamma,\loc}(\Omega;\K)$ for $\K=\K(c)$. Those spaces are then referred to as \emph{locally weighted H\"older spaces}.

\begin{Rem} \label{Rem:locally_weighted_hoelder_spaces}
Obviously, for every choice of the parameters, $C^{\ell,\alpha}_{\gamma,\loc}(\Omega)$ contains $C^{\ell,\alpha}(\Omega)$ as a linear subspace, but it also contains functions $g$ whose local H\"older semi-norms $\abs{g}_{C^{\ell,\alpha}(K)}$ grow to infinity as the distance $\delta_K$ of $K\subset\subset \Omega$ to the boundary tends to zero. However, this possible blow-up is controlled by the parameter $\gamma$.
Moreover, in the Appendix we show that the intersection of $C^{\ell,\alpha}_{\gamma,\loc}(\Omega)$ with some Besov space is a Banach space with respect to the canonical norm; see \autoref{Prop:intersect}. Finally, we want to mention that the spaces $C^{\ell,\alpha}_{\gamma,\loc}(\Omega)$ are monotone in $\gamma$, meaning that $C^{\ell,\alpha}_{\gamma,\loc}(\Omega) \subseteq C^{\ell,\alpha}_{\mu,\loc}(\Omega)$ for $\gamma \leq \mu$. This can be seen by checking that $\delta_K^{\mu} = (\delta_K/C)^{\mu} C^{\mu} \leq (\delta_K/C)^{\gamma} C^{\mu} = \delta_K^{\gamma} C^{\mu - \gamma}$ for some universal constant $C\geq 1$ (e.g., $C = \max\{1,\diam(\Omega)\}$), thus $\abs{\cdot}_{C^{\ell,\alpha}_{\mu,\loc}} \leq \abs{\cdot}_{C^{\ell,\alpha}_{\gamma,\loc}}$.
\end{Rem}

For the sake of completeness, we mention here that (as usual) the set of all infinitely often (strongly) differentiable functions with compact support in $\Omega$ will be denoted by $C_0^\infty(\Omega)$ or $\D(\Omega)$. For its dual space we write $\D'(\Omega)$.
Once more, these definitions apply likewise when $\Omega$ is replaced by some compact set $K$.

\subsection{Weakly differentiable functions: Sobolev spaces}
Assume $\Omega\subseteq\R^d$ to be either $\R^d$ itself, or some bounded domain. 
Given $0<p\leq\infty$ the \emph{Lebesgue spaces} $L_p(\Omega)$ consist of all (equivalence classes of real-valued) measurable functions~$g$ on $\Omega$ for which the (quasi-)norm
\begin{equation*}
	\norm{g \sep L_p(\Omega)}
	= \begin{cases}
		\left( \displaystyle\int_\Omega \abs{g(x)}^p \d x \right)^{1/p} & \text{if } p<\infty,\\
		\esssup_{x\in\Omega} \abs{g(x)} & \text{if } p=\infty
	\end{cases}
\end{equation*}
is finite.

Moreover, for $1 \leq p < \infty$ and $\ell\in\N_0$, let 
\begin{equation*}
	W^\ell(L_p(\Omega))
	= \left\{ g \in L_p(\Omega) \sep \norm{g \sep W^\ell(L_p(\Omega))} = \sum_{\abs{\nu} \leq \ell} \norm{D^\nu g \sep L_p(\Omega)}<\infty \right\}
\end{equation*}
denote the classical \emph{Sobolev spaces} on $\Omega$,
where $D^\nu$ are the weak partial derivatives of order 
$\nu\in\N_0^d$.
For fractional smoothness parameters $s=\ell+\beta>0$ (with $\ell\in\N_0$ and $0<\beta<1$) we extend the definition in the usual way by setting
\begin{equation*}
	W^s(L_p(\Omega))
	= \left\{ g \in W^\ell(L_p(\Omega)) \sep \norm{g \sep W^s(L_p(\Omega))} <\infty \right\},
\end{equation*}
where here the norm is given by $\norm{g \sep W^s(L_p(\Omega))} = \norm{g \sep W^\ell(L_p(\Omega))} + \abs{g}_{W^s(L_p(\Omega))}$ and
\begin{equation*}
	\abs{g}_{W^s(L_p(\Omega))} 
	= \left( \sum_{\abs{\nu}=\ell} \int_\Omega \int_\Omega \frac{\abs{D^\nu g(x) - D^\nu g(y)}^p}{\abs{x-y}^{d+\beta\, p}} \d x\d y \right)^{1/p}
\end{equation*}
denotes the common Sobolev semi-norm on $\Omega$. 

Furthermore, for $s > 0$ and $1<p<\infty$, let us denote the closure of $C_0^\infty(\Omega)$ in the norm of $W^s(L_p(\Omega))$ by $W^s_0(L_p(\Omega))$. 
Then we define $W^{-s}(L_{p'}(\Omega))$ to be the dual space of $W^s_0(L_{p}(\Omega))$, where $p'$ is determined by the relation $1/p + 1/p' = 1$. 

For a detailed discussion of the scale of Banach spaces $W^s(L_p(\Omega))$, $s\in\R$, we refer to standard textbooks such as \cite{Ada1975, Tri1983} and the references given therein.

\subsection{Generalized smoothness: Besov spaces}\label{sec:besov_def}
A more advanced way to measure the smoothness of functions is provided by the framework of Besov spaces which essentially generalizes the concept of Sobolev spaces introduced above. 
Besov spaces can be defined in various ways 
which (for a large range of the parameters involved) lead to equivalent descriptions; cf.~\cite{BerLof1976,DahNovSic2006b,Tri1983,Tri2006}.
For our purposes the following approach based on iterated differences seems to be the most reasonable one, since it provides an entirely \emph{intrinsic} definition when dealing with Lipschitz domains (i.e., domains which possess a  Lipschitz boundary; cf.\ \cite[Def.~1.103]{Tri2006}). We refer, e.g., to \cite{Coh2003,Dev1998,DevPop1988,DevSha1984,DevSha1993}. 

In the following let $\Omega\subseteq\R^d$ be either $\R^d$ itself, or some bounded Lipschitz domain. Moreover, let $r\in\N$ and $h \in \R^d$. Then $\Omega_{r,h}$ denotes the set of all $x \in \Omega$ such that the line segment $[x,x+rh]$ belongs to $\Omega$. 
Moreover, for functions $g$ on $\Omega$ the \emph{iterated difference} of order $r$ with step size $h$ is recursively given by
\begin{equation*}
	\Delta^1_h(g,x) = g(x+h)-g(x)
	\qquad \text{and} \qquad
	\Delta_h^r(g,x) = \Delta_h^1(\Delta_h^{r-1}(g,\cdot),x), 
	\quad r\geq 2,
\end{equation*}
for every $x\in\Omega_{r,h}$.
It is easily verified that 
\begin{equation*}
	\Delta^r_h(g,x) = \sum^r_{k=0} (-1)^{r-k} \, {r \choose k} \, g(x + kh)
	\qquad \text{for all} \quad r\in\N, \; h\in\R^d, \; x\in \Omega_{r,h}.
\end{equation*}
Those differences can be used to quantify smoothness: For $0<p\leq \infty$ and every $g\in L_p(\Omega)$ let
\begin{equation*}
	\omega_r(g,t,\Omega)_p
	= \sup_{h\in\R^d, \abs{h}\leq t} \norm{ \Delta^r_h(g,\cdot) \sep L_p(\Omega_{r,h})}, 
	\qquad t > 0,
\end{equation*}
denote the \emph{modulus of smoothness} of order $r$.
It is well-known that $\omega_r(g,t,\Omega)_p \rightarrow 0$ monotonically as $t$ tends to zero and the faster this convergence the smoother is $g$.

Now let $s=\ell+\beta > 0$ with $\ell\in\N_0$ and $0\leq \beta <1$. Then, for $0 < p,q \leq \infty$, the \emph{Besov space} $B^{s}_q(L_p(\Omega))$ is defined as the collection of all $g \in L_p(\Omega)$ for which the semi-norm
\begin{equation}\label{def:Besov_semi}
	\abs{g}_{B^{s}_q(L_p(\Omega))} 
	= \begin{cases}
		\left( \displaystyle\int_0^\infty \left[t^{-s} \, \omega_r(g,t,\Omega)_p \right]^q \frac{\d t}{t}\right)^{1/q} &\text{if } q<\infty,\\
		\sup_{t>0}\limits \, t^{-s} \, \omega_r(g,t,\Omega)_p &\text{if } q=\infty,
	\end{cases}
\end{equation}
with $r\geq \ell+1$ is finite. Endowed with the canonical (quasi-)norm
\begin{equation*}
	\norm{g \sep B^{s}_q(L_p(\Omega))}
	= \norm{g \sep L_p(\Omega)} + \abs{g}_{B^{s}_q(L_p(\Omega))} 
\end{equation*}
these spaces turn out to be quasi-Banach spaces (and Banach spaces if $\min\{p,q\}\geq 1$).
Roughly speaking, with $\norm{g \sep B^{s}_q(L_p(\Omega))}$ we can control all (weak) partial derivatives $D^\nu g$ up to the order $s$, measured in $L_p(\Omega)$. Since the influence of the additional \emph{fine index} $q$ is neglectable for many applications, we will mainly focus on the \emph{smoothness parameter} $s$, as well as on the \emph{integrability index} $p$, and simply set $q=p$ in what follows.

\begin{Rem}\label{rem:Besov_prop}
Some comments are in order:
\begin{itemize}
	\item[(i)] We note that different choices of $r\geq \floor{s}+1$ in \link{def:Besov_semi} lead to equivalent (quasi-)norms. The same is true when we restrict the range for $t$ in \link{def:Besov_semi} to the interval $(0,1)$.
	\item[(ii)] The scale of Besov spaces as defined above is well-studied. In particular, sharp assertions on embeddings, interpolation and duality properties, characterizations in terms of various building blocks (e.g., atoms, local means, quarks, or wavelets) and best $n$-term approximation results are known; see, e.g., \cite{DahNovSic2006b,Dev1998,DevSha1993,HanSic2011}. Many of them can also be shown using the Fourier analytic definition of $B^s_q(L_p(\Omega))$ as spaces of (restrictions of) tempered distributions \cite{FraJaw1990,Tri1983,Tri2006}. It is known \cite{Dis2003,Schn2011,Tri2006} that both definitions coincide in the sense of equivalent (quasi-)norms if
\begin{equation}\label{def:sigma_p}
	s > \sigma_p = d \cdot \max\!\left\{\frac{1}{p}-1,0\right\}.
\end{equation}
	\item[(iii)] The demarcation line for embeddings of Besov spaces into $L_p(\Omega)$, $1<p<\infty$, is given by
\begin{equation}\label{eq:tau_scale}
	\frac{1}{\tau} = \frac{\sigma}{d} + \frac{1}{p}.
\end{equation}
Every Besov space with smoothness and integrability indices corresponding to a point
above that line is continuously embedded into $L_p(\Omega)$ (regardless of the fine index $q$). The points below this line never embed into $L_p(\Omega)$. For spaces $B^\sigma_q(L_\tau(\Omega))$ with $(\sigma,\tau)$ that satisfy \link{eq:tau_scale} some care is needed. However, if $q=\tau$, then the embedding still holds.
	Observe that \link{eq:tau_scale} exactly coincides with the adaptivity scale of Besov spaces we are interested in.
	\item[(iv)] Besov spaces are closely related to Sobolev spaces. Indeed, it has been shown that for bounded Lipschitz domains $\Omega$, $1\leq p <\infty$, and $0<s\notin\N$ the space $B^s_p(L_p(\Omega))$ coincides with $W^s(L_p(\Omega))$ in the sense of equivalent norms; see, e.g., \cite[Theorem~6.7]{DevSha1993}. 
	Using the fact that $X^s(L_p(\Omega)) \hookrightarrow X^{s-\varepsilon}(L_p(\Omega))$ for $X\in\{B_p,W\}$ and arbitrary small $\varepsilon>0$ we thus have
	\begin{equation*}
		W^{s+\varepsilon}(L_p(\Omega)) \hookrightarrow B_p^s(L_p(\Omega)) \hookrightarrow W^{s-\varepsilon}(L_p(\Omega))
	\end{equation*}
	for all $1\leq p < \infty$ and every $s>\varepsilon>0$.
	\item[(v)] For every bounded Lipschitz domain $\Omega\subset\R^d$ there exists a linear extension operator 
\begin{equation*}
	\E_\Omega \colon B^s_q(L_p(\Omega)) \rightarrow B^s_q(L_p(\R^d))
\end{equation*}
which is simultaneously bounded for all parameters that satisfy \link{def:sigma_p}; cf.\ \cite{Ryc1999}.
Moreover, $\E_\Omega$ is local in the sense that $\supp (\E_\Omega u)$ is contained in some bounded neighborhood of~$\Omega$; see \cite{DahNovSic2006b}.
\end{itemize}
\end{Rem}

\subsection{Wavelet characterization of Besov spaces}\label{subsec:wavelet_char}
Under suitable conditions on the parameters involved it is possible to characterize Besov spaces by means of wavelet decompositions \cite{Dau1992,HedNet2007,Mey1992,Tri2006}. 
These characterizations are one of the most important ingredients of wavelet analysis. 
In particular, they provide the basis for several numerical applications such as preconditioning and the design of adaptive algorithms. We refer to \cite{Coh2003, CohDahDev2001, DahDahDev1997} for details.
Moreover, the resulting (quasi-)norm equivalences provide a powerful tool which allows to prove continuous embeddings such as the one stated in \autoref{theorem:embedding} in \autoref{sec:embedding} below.

To start with, we recall some basic assertions related to expansions w.r.t.\ Daubechies wavelets. We essentially follow the lines of \cite{DahDev1997}: Let $\{D_m \sep m \in \N\}$ denote the univariate family of compactly supported Daubechies wavelets \cite{Dau1988,Dau1992}. 
We remind the reader that $D_m$ has~$m$ vanishing moments and the smoothness of these functions increases without bound as $m$ tends to infinity. 
So, let us fix an arbitrary value of $m$ and let $\psi^0 = \phi_m$ denote the univariate scaling function which generates the wavelet $\psi^1 = D_m$. Furthermore, by $E$ we denote the non-zero vertices of the unit cube $[0,1]^d$. Then, in dimension $d$, the set 
\begin{equation*}
	\Psi = \Psi(d) = \left\{\psi^e = \bigotimes_{n=1}^d \psi^{e_n} \sep e=(e_1,\ldots,e_d)\in E\right\}
\end{equation*}
of $2^d - 1$ (tensor product) functions generates (by shifts and dilates) an orthonormal wavelet basis for $L_2(\mathbb{R}^d)$ as follows: If
\begin{equation*}
	\I = \I(\mathbb{R}^d) = \left\{I_{j,k} = 2^{-j} k + 2^{-j} [0,1]^d \, \sep \, k \in \Z^d, j\in\Z \right\}
\end{equation*}
denotes the set of all dyadic intervals in $\R^d$, then the basis consists of all functions of the form
\begin{align}\label{wavelet_basis}
	\eta_I = \eta_{j,k} = 2^{j\,d/2} \, \eta(2^j \cdot -k) 
	\qquad \text{with} \quad 
	I = I_{j,k} \in \I, 
	\quad k \in \Z^d, 
	\quad j \in \Z, 
	\quad \text{and} \quad 
	\eta \in \Psi.
\end{align}
In view of our application below, we remark that there exists some open cube $Q\subset\R^d$, centered at the origin with sides parallel to the coordinate axes, such that $\supp(\eta) \subset Q$ for all $\eta\in\Psi$. Accordingly, all basis functions \link{wavelet_basis} satisfy $\supp(\eta_I) \subset Q(I) = 2^{-j}k + 2^{-j}Q$, where
\begin{equation}\label{def:QI}
	\abs{Q(I)} \sim \abs{I} = 2^{-j \, d}
	\qquad \text{and} \qquad
	Q(I) \subset B(I)=B_{2^{-(j+1)}\diam(Q)}(2^{-j}k), \quad I=I_{j,k}\in\I.
\end{equation}
For every $1 < q < \infty$ the system defined in \link{wavelet_basis} also forms an unconditional basis for $L_q(\R^d)$. Hence, for those $q$ each $g \in L_q(\R^d)$ possesses a wavelet expansion
\begin{align}\label{wavelet_expansion}
	g = \sum_{I \in \I} \sum_{\eta \in \Psi} \distr{g}{\eta_I} \eta_I 
\end{align}
which converges in $L_q(\R^d)$.

For our purposes it is convenient to slightly modify this decomposition.
Therefore let~$S_0$ be the closure of all finite linear combinations of integer shifts of $\bigotimes_{n=1}^d \phi_m$ in $L_2(\R^d)$ and let~$P_0$ denote the orthogonal projector which maps $L_2(\R^d)$ onto $S_0$. Then, for every $1< q < \infty$, the operator $P_0$ can be extended to a projector on $L_q(\R^d)$ and in \link{wavelet_expansion} we can restrict ourselves to those $\eta_I$ for which 
\begin{equation*}
	I \in \I^+ 
	= \I^+(\R^d)
	= \{I\in\I(\R^d) \sep \abs{I}\leq 1\},
\end{equation*}
i.e., to wavelets corresponding to levels $j\in\N_0$. Moreover, we shall renormalize our wavelets and set
\begin{equation*}
	\eta_{I,p} = \abs{I}^{1/2 - 1/p} \eta_I
	\qquad \text{for all} \quad I\in\I^+, \quad \eta\in\Psi, \quad \text{and} \quad 0<p<\infty,
\end{equation*}
such that $\norm{\eta_{I,p} \sep L_p(\R^d)} = \norm{\eta \sep L_p(\R^d)}$ does not depend on $I$. Incorporating these conventions, from \link{wavelet_expansion} we conclude that every $g\in L_q(\R^d)$, $1<q<\infty$, can be expanded as 
\begin{align}
	g 
	&= P_0(g) + \sum_{I \in \I^+} \sum_{\eta \in \Psi} \distr{g}{\eta_I} \eta_I \nonumber\\
	&= P_0(g) + \sum_{I \in \I^+} \sum_{\eta \in \Psi} \distr{g}{\eta_{I,p'}} \eta_{I,p},
	 \label{wavelet_expansion_2}
\end{align}
where $p'$ satisfies $1/p'=1-1/p$.

\begin{lemma} \label{Prop:wavelet_representation}
Let $d\in\N$, $0 < p < \infty$, and $\sigma_p<s< r \in \N$. 
Moreover, choose $m\in\N$ such that $\phi_m, D_m \in C^r(\R)$. 
Then a function $g$ belongs to the Besov space $B^s_p(L_p(\mathbb{R}^d))$ if and only if \link{wavelet_expansion_2} holds with
\begin{equation}\label{equiv_norm_1}
	\norm{P_0(g) \sep L_p(\R^d)} + \left( \sum_{I \in \I^+} \sum_{\eta \in \Psi} \abs{I}^{-s\, p/d} \abs{\distr{g}{\eta_{I,p'}}}^p \right)^{1/p} < \infty.
\end{equation}
Furthermore, (\ref{equiv_norm_1}) provides an equivalent (quasi-)norm for $B^s_p(L_p(\R^d))$.
\end{lemma}
The proof of this assertion is quite standard. For the case of Banach spaces ($p\geq 1$) it can be found, e.g., in \cite{Mey1992}. For the quasi-Banach case $0<p<1$ we refer to \cite{Kyr1996}. Similar assertions can also be found in \cite{Tri2006}.

\begin{Rem}\label{rem:wavelet_char_domain}
We stress the point that due to $s > \sigma_p$ every $g\in B^s_p(L_p(\R^d))$ belongs to some $L_q(\R^d)$, $1<q<\infty$, such that \link{wavelet_expansion_2} is well-defined; see \autoref{rem:Besov_prop}(iii). Moreover, we can use the extension operator $\E_\Omega$ described in \autoref{rem:Besov_prop}(v) to obtain similar norm equivalences for functions in $B^s_p(L_p(\Omega))$, where $\Omega\subset\R^d$ is a bounded Lipschitz domain.
\end{Rem}

As mentioned already in the introduction, we are particularly interested in Besov spaces $B^{\sigma}_{\tau}(L_{\tau}(\Omega))$ within the adaptivity scale of $L_p(\Omega)$, $1<p<\infty$, i.e., spaces with parameters that satisfy \link{eq:tau_scale}. Therefore, we specialize \autoref{Prop:wavelet_representation} for the corresponding spaces on $\R^d$:

\begin{Prop} \label{Prop:wavelet_representation_tau}
Let $d\in\N$, $1 < p < \infty$, as well as $0<\sigma< r \in \N$, and $\tau = (\sigma/d + 1/p)^{-1}$. Moreover, choose $m\in\N$ such that $\phi_m, D_m \in C^r(\R)$.
Then a function $g$ belongs to the Besov space $B^{\sigma}_{\tau}(L_{\tau}(\R^d))$ if and only if
\begin{equation*}
	g = P_0(g) + \sum_{I \in \I^+} \sum_{\eta \in \Psi} \distr{g}{\eta_{I,p'}} \eta_{I,p}
\end{equation*}
with
\begin{equation}\label{equiv_norm_2}
	\norm{P_0(g) \sep L_\tau(\R^d)} + \left( \sum_{I \in \I^+} \sum_{\eta \in \Psi} \abs{\distr{g}{\eta_{I,p'}}}^\tau \right)^{1/\tau} < \infty
\end{equation}
and (\ref{equiv_norm_2}) provides an equivalent (quasi-)norm for $B^{\sigma}_{\tau}(L_{\tau}(\R^d))$.
\end{Prop}
\begin{proof}
Observe that $\eta_{I,\tau'} = \abs{I}^{1/p' - 1/\tau'} \eta_{I,p'}$ implies $\abs{I}^{-\sigma \tau / d} \abs{\distr{g}{\eta_{I,\tau'}}}^{\tau} = \abs{\distr{g}{\eta_{I,p'}}}^{\tau}$.
Then the proof easily follows from \autoref{Prop:wavelet_representation}.
\end{proof}


\section{A general embedding} \label{sec:embedding}
In this section we prove that, under some growth conditions on the local H\"older semi-norm, the intersection $B^s_p(L_p(\Omega)) \cap C^{\ell,\alpha}_{\gamma, \loc}(\Omega)$ is continuously embedded into certain Besov spaces $B^{\sigma}_{\tau}(L_{\tau}(\Omega))$. 

\begin{theorem} \label{theorem:embedding}
For $d\in\N$ with $d\geq 2$, let $\Omega\subset\R^d$ denote some bounded Lipschitz domain.
Moreover, let $s>0$ and $1<p<\infty$, as well as $\ell\in\N_0$, $0<\alpha\leq 1$, and $0<\gamma < \ell + \alpha + 1/p$. If we define
\begin{equation}\label{def:max_alpha}
	\sigma^* 
	= \begin{dcases}
		\ell + \alpha
			&\;\text{if}\quad 
			\qquad\qquad 0 < \gamma < \frac{\ell+\alpha}{d} + \frac{1}{p},\\
		\frac{d}{d-1} \left( \ell + \alpha + \frac{1}{p} - \gamma \right)
			&\;\text{if}\quad\,
			\frac{\ell+\alpha}{d} + \frac{1}{p} \leq \gamma < \ell + \alpha + \frac{1}{p},
	\end{dcases}
\end{equation}
then for all
\begin{equation}\label{range:thm}
	0 < \sigma < \min\!\left\{ \sigma^*, \frac{d}{d-1}\, s \right\}
	\qquad \text{and} \qquad 
	\frac{1}{\tau} = \frac{\sigma}{d} + \frac{1}{p}
\end{equation}
we have the continuous embedding
\begin{equation*}
	B^{s}_p(L_p(\Omega)) \; \cap \; C^{\ell,\alpha}_{\gamma,\loc}(\Omega)
	\; \hookrightarrow \; B^{\sigma}_{\tau}(L_{\tau}(\Omega)),
\end{equation*}
i.e., for all $u\in B^{s}_p(L_p(\Omega)) \cap C^{\ell,\alpha}_{\gamma,\loc}(\Omega)$ it holds
\begin{equation}\label{ineq:norm_est}
	\norm{u \sep B^{\sigma}_{\tau}(L_{\tau}(\Omega))} 
	\lesssim \max\left\{ \norm{u \sep B^{s}_p(L_p(\Omega))}, \abs{u}_{C^{\ell,\alpha}_{\gamma,\loc}} \right\}.
\end{equation}
\end{theorem}

Let us briefly comment on \autoref{theorem:embedding} before we give its proof:
From the theory of function spaces it is well-known that (standard) embeddings between Besov spaces, e.g., 
\begin{equation*}
	B^{s}_p(L_p(\Omega)) \hookrightarrow B^{\sigma}_{\tau}(L_{\tau}(\Omega)),
\end{equation*}
are valid \emph{only if} the regularity of the target space is at most as large as the smoothness of the space we start from, i.e., only if $\sigma \leq s$. 
\autoref{theorem:embedding} now states that, under suitable assumptions on the parameters involved, exploiting the additional information on locally weighted H\"older regularity (encoded by the membership of $u$ in $C^{\ell,\alpha}_{\gamma,\loc}(\Omega)$) enables us to prove that functions from $B^{s}_p(L_p(\Omega))$ indeed possess a higher-order Besov regularity $\sigma > s$ measured in the adaptivity scale corresponding to $L_p(\Omega)$. Since $B^{s}_p(L_p(\Omega))$ almost equals the Sobolev space $W^s(L_p(\Omega))$ (cf.\ \autoref{rem:Besov_prop}(iv)) this shows that 
approximating $u \in W^s(L_p(\Omega)) \; \cap  \; C^{\ell,\alpha}_{\gamma,\loc}(\Omega)$ in an adaptive way is justified whenever $\sigma^*$ defined by \link{def:max_alpha} is larger than $s$. 
At this point we remark that $\sigma^*$ is a continuous piecewise linear function of $\gamma\in(0,\ell+\alpha+1/p)$ which decreases to zero when $\gamma$ approaches its upper bound.
Hence, in any case $0<\sigma^*\leq \ell+\alpha$.
Thus, for a fixed value of $s$, the maximal regularity $d/(d-1)\cdot s$ is achieved if $\ell+\alpha$ is sufficiently large and $\gamma$ is small enough.

The proof of \autoref{theorem:embedding} given below is inspired by ideas first given in \cite{DahDev1997}. Due to extension arguments in conjunction with the wavelet characterization of Besov spaces on $\R^d$ (see \autoref{rem:wavelet_char_domain}) it suffices to find suitable estimates for the wavelet coefficients $\distr{u}{\eta_{I,p'}}$, $I\in\I^+$, $\eta\in\Psi$, which then imply \link{ineq:norm_est}. The contribution of (the relatively small number of) wavelets supported in the vicinity of the boundary of $\Omega$ (\emph{boundary wavelets}) can be bounded in terms of the norm of $u$ in $B^s_p(L_p(\Omega))$. 
Here the restriction $\sigma < s \cdot d/(d-1)$ comes in. 
The coefficients corresponding to the remaining \emph{interior wavelets} can be upper bounded by the semi-norm of $u$ in $C^{\ell,\alpha}_{\gamma,\loc}$ using a Whitney-type argument which then gives rise to the restriction $\sigma<\sigma^*$. The detailed proof reads as follows:

\begin{proof}[Proof (of \autoref{theorem:embedding})]
\emph{Step 1.} 
Let $u\in B^{s}_p(L_p(\Omega)) \cap C^{\ell,\alpha}_{\gamma,\loc}(\Omega)$. 
Since for $1<p<\infty$ it is $\sigma_p=0$ and $s > 0$, every such $u$ can be extended to some $\E_\Omega u \in B^s_p(L_p(\R^d))$; see \autoref{rem:Besov_prop}(v).
In particular, $\E_\Omega u\in L_p(\R^d)$ such that it can be written as
\begin{equation*}
	\E_\Omega u
	= P_0(\E_\Omega u) + \sum_{(I,\eta)\in\I^+\times\Psi} \distr{\E_\Omega u}{\eta_{I,p'}} \eta_{I,p}.
\end{equation*}
Here the $\eta_I$ form a system of Daubechies wavelets \link{wavelet_basis}, where $m\in\N$ is chosen such that $m>\ell$ and $\phi_m, D_m\in C^r(\R)$ for some $r\in\N$ with $r > \max\{\sigma,s\}$; see \autoref{subsec:wavelet_char} for details.
We restrict the latter expansion and consider only those wavelets for which $(I,\eta)$ belongs to
\begin{equation*}
	\Lambda 
	= \bigcup_{j\in\N_0} \Lambda_j,
	\quad \text{ where we set } \quad
	\Lambda_j = \left\{ (I,\eta)\in \I^+\times\Psi \sep B_{c}(I) \cap \Omega \neq \emptyset \; \text{ and } \; \abs{I}=2^{-jd} \right\}.
\end{equation*}
Therein $B_c(I)$ denotes the ball $B(I)$ (see \link{def:QI}) concentrically expanded by the factor $c>1$ which we used to define the class $C^{\ell,\alpha}_{\gamma,\loc}(\Omega)$; cf.\ \autoref{subsec:Hoelder}. Note that thus $\supp(\eta_I)\subset B_c(I)$ for all $I$ and $\eta$.
Next we split up the index sets $\Lambda_j$ once more and write
\begin{equation*}
	\Lambda_j 
	= \bigcup_{n\in\N_0} \Lambda_{j,n}
	\quad \text{ with } \quad
	\Lambda_{j,n} = \left\{ (I_{j,k},\eta)\in\Lambda_j \sep n\, 2^{-j} \leq \dist\!\left( 2^{-j}k, \partial\Omega \right) < (n+1)\, 2^{-j} \right\},
\end{equation*}
for every dyadic level $j\in\N_0$.
Note that, due to the boundedness of $\Omega$, there exists an absolute constant $C_1$ such that $\Lambda_{j,n}=\emptyset$ for all $j\in\N_0$ and $n>C_1\,2^j$. For example, we may take $C_1=\max\{\diam(\Omega), c\, \diam(Q)\}$. Moreover, our assumption that $\Omega$ is a bounded Lipschitz domain ensures that all remaining index sets satisfy at least $\abs{\Lambda_{j,n}} \lesssim 2^{-j(d+1)}$.
Finally, we note that all balls $B_c(I)$ corresponding to $(I,\eta)\in\Lambda_{j,n}$ with $j\in\N_0$ and $n$ strictly larger than $C_0 = \lceil c\, \diam(Q)/2 \rceil$ are completely contained in $\Omega$.
These considerations justify the disjoint splitting $\Lambda = \left( \bigcup_{j\in\N_0} \Lambda^\mathrm{bnd}_j \right) \cup \left( \bigcup_{j\in\N_0} \Lambda^\mathrm{int}_j \right)$, where
\begin{equation*}
	\Lambda^\mathrm{bnd}_j = \bigcup_{n=0}^{C_0} \Lambda_{j,n}
	\qquad \text{and} \qquad
	\Lambda^\mathrm{int}_j = \bigcup_{n=C_0+1}^{C_1\,2^j} \Lambda_{j,n}
\end{equation*}
correspond to the sets of boundary and interior wavelets at level $j\in\N_0$, respectively.
Observe that then
$\tilde{u}=u_0+u_1+u_2$, defined by
\begin{equation*}
	u_0 = P_0(\E_\Omega u), \quad 
	u_1 = \sum_{j\in\N_0}\sum_{(I,\eta)\in \Lambda^\mathrm{bnd}_j} \distr{\E_\Omega u}{\eta_{I,p'}} \eta_{I,p}, \quad \text{and} \quad
	u_2 = \sum_{j\in\N_0}\sum_{(I,\eta)\in \Lambda^\mathrm{int}_j} \distr{u}{\eta_{I,p'}} \eta_{I,p},
\end{equation*}
is an extension of $u$ as well, i.e., it satisfies $\tilde{u}\big|_{\Omega}=u$. In Step 2--4 below we will show that for the adaptivity scale $\tau=\left( \sigma/d+1/p\right)^{-1}$ it holds
\begin{align}
	&\norm{u_0 \sep B^{\sigma}_\tau(L_\tau(\R^d))} 
		\lesssim \norm{P_0(\E_\Omega u) \sep L_p(\R^d)} 		&\text{if }\:& 0<\sigma,
		\label{est:projection}\\
	&\norm{u_1 \sep B^{\sigma}_\tau(L_\tau(\R^d))} 
		\lesssim \left[ \sum_{j\in\N_0}\sum_{(I,\eta)\in \Lambda^\mathrm{bnd}_j} \!\! \abs{I}^{-s\,p/d} \abs{\distr{\E_\Omega u}{\eta_{I,p'}}}^p \right]^{1/p} &\text{if }\:& 0 < \sigma< \frac{d}{d-1}\,s, \text{ and}
		\label{est:boundary}\\
	&\norm{u_2 \sep B^{\sigma}_\tau(L_\tau(\R^d))} 
		\lesssim \abs{u}_{C^{\ell,\alpha}_{\gamma,\loc}} &\text{if }\:& 0<\sigma<\sigma^*.
		\label{est:interior}
\end{align}
Suppose we already know that those relations hold for all $\sigma$ and $\tau$ that satisfy \link{range:thm}.
Then we can extend the index set in \link{est:boundary} from $\bigcup_{j\in\N_0} \Lambda^\mathrm{bnd}_j$ to $\I^+\times\Psi$ and the wavelet characterization of $\E_\Omega u\in B^s_p(L_p(\R^d))$ (cf. \autoref{Prop:wavelet_representation}) together with the continuity of $\E_\Omega$ implies
\begin{equation}\label{est:ext}
	\norm{u_0+u_1 \sep B^{\sigma}_\tau(L_\tau(\R^d))}
	\lesssim \norm{\E_\Omega u \sep B^s_p(L_p(\R^d))} 
	\sim \norm{u \sep B^s_p(L_p(\Omega))}
\end{equation}
which is finite due to our assumptions.
Therefore, the special choice $g=\tilde{u}=(u_0+u_1)+u_2$, in conjunction with \link{est:interior} and \link{est:ext}, yields the desired estimate
\begin{align*}
	\norm{u \sep B^{\sigma}_\tau(L_\tau(\Omega))}
	&\sim \inf 
		\!\left\{ \norm{g \sep B^{\sigma}_\tau(L_\tau(\R^d))} \sep g\in B^{\sigma}_\tau(L_\tau(\R^d)) \; \text{ with } \; g\big|_{\Omega}=u \right\}\\
	&\lesssim  \norm{u_0 + u_1 \sep B^{\sigma}_\tau(L_\tau(\R^d))} + \norm{u_2 \sep B^{\sigma}_\tau(L_\tau(\R^d))}\\
	&\lesssim \max\left\{ \norm{u \sep B^s_p(L_p(\Omega))}, \abs{u}_{C^{\ell,\alpha}_{\gamma,\loc}} \right\}.
\end{align*}
This proves \autoref{theorem:embedding} since $u\in B^s_p(L_p(\Omega))$ with $s>0=\sigma_p$ particularly implies that $u\in L_p(\Omega) \hookrightarrow L_\tau(\Omega)$, due to $\tau < p$ and the boundedness of $\Omega$. Hence, $u\in B^{\sigma}_\tau(L_\tau(\Omega))$.

\emph{Step 2 (Estimate for $u_0$).}
To show the bound on the projection onto the coarse levels let $\tau=\left( \sigma/d+1/p\right)^{-1}$ and $\sigma>0$. We note that $u_0 \bot \eta_{I,p'}$ for all $I\in\I^+$ and $\eta\in\Psi$, i.e., $u_0=P_0(u_0)$. Moreover, by definition, this equals $P_0(\E_\Omega u)$ which has compact support in~$\R^d$ since $\E_\Omega$ is local; see \autoref{rem:Besov_prop}(v).
\autoref{Prop:wavelet_representation_tau}, i.e., the wavelet characterization of $B^{\sigma}_\tau(L_\tau(\R^d))$, therefore gives
\begin{equation*}
	\norm{u_0 \sep B^{\sigma}_\tau(L_\tau(\R^d))} 
	\sim \norm{P_0(\E_\Omega u) \sep L_\tau(\R^d)}
	\lesssim \norm{P_0(\E_\Omega u) \sep L_p(\R^d)},
\end{equation*}
due to $\tau < p$. That is, we have shown \link{est:projection}.

\emph{Step 3 (Estimate for $u_1$).}
Here we establish the bound on the contribution of all wavelets near $\partial\Omega$.
To this end, assume again that $\tau=\left( \sigma/d+1/p\right)^{-1}$ with $\sigma>0$. We fix $j\in\N_0$ for a moment and apply H\"older's inequality (with $q=p/\tau>1$) to estimate
\begin{align*}
	\sum_{(I,\eta)\in \Lambda^\mathrm{bnd}_j} \abs{\distr{\E_\Omega u}{\eta_{I,p'}}}^\tau
	&\leq \abs{\Lambda^\mathrm{bnd}_j}^{1-\tau/p} \left( \sum_{(I,\eta)\in\Lambda^\mathrm{bnd}_j} \abs{\distr{\E_\Omega u}{\eta_{I,p'}}}^p \right)^{\tau/p} \\
	& \lesssim 2^{j(d-1)(1-\tau/p)} \, 2^{-j\, s\, \tau} \left( \sum_{(I,\eta)\in\Lambda^\mathrm{bnd}_j} \abs{I}^{-s \, p/d} \abs{\distr{\E_\Omega u}{\eta_{I,p'}}}^p \right)^{\tau/p}.
\end{align*}
Taking the sum over all levels $j$ and using H\"older's inequality once more (with the same $q$), we find
\begin{align}
	&\sum_{j\in\N_0}\sum_{(I,\eta)\in\Lambda^\mathrm{bnd}_j} \abs{\distr{\E_\Omega u}{\eta_{I,p'}}}^\tau \label{final_est_boundary}\\
	&\qquad \lesssim \left( \sum_{j\in\N_0} \left[2^{ (d-1)-s\,\tau /(1-\tau/p) } \right]^j \right)^{1-\tau/p} \left( \sum_{j\in\N_0}\sum_{(I,\eta)\in\Lambda^\mathrm{bnd}_j} \abs{I}^{-s \, p/d} \abs{\distr{\E_\Omega u}{\eta_{I,p'}}}^p \right)^{\tau/p} \nonumber\\
	&\qquad \lesssim \left( \sum_{j\in\N_0}\sum_{(I,\eta)\in\Lambda^\mathrm{bnd}_j} \abs{I}^{-s \, p/d} \abs{\distr{\E_\Omega u}{\eta_{I,p'}}}^p \right)^{\tau/p}, \nonumber
\end{align}
provided that we additionally assume 
\begin{equation*}
	\sigma < \frac{d}{d-1} \, s,
\end{equation*}
since this condition is equivalent to $1/\tau < s/(d-1) + 1/p$ which in turn holds if and only if
$(d-1)-s\,\tau /(1-\tau/p)<0$.
Finally, the structure of $u_1$ together with \autoref{Prop:wavelet_representation_tau} shows that the quantity \link{final_est_boundary} is equivalent to $\norm{u_1 \sep B^{\sigma}_\tau(L_\tau(\R^d))}^\tau$ such that \link{est:boundary} follows.

\emph{Step 4 (Estimate for $u_2$).}
We are left with the proof of \link{est:interior}, i.e., the bound for the interior wavelets indexed by $(I,\eta)\in \bigcup_{j\in\N_0} \Lambda^\mathrm{int}_j$.
Recall that $\eta_{I,p'}$ is orthogonal to every polynomial $\P$ of total degree strictly less than $m$. Therefore, for all $(I,\eta)$ under consideration, 
\begin{align*}
	\abs{\distr{u}{\eta_{I,p'}}} 
	= \abs{\distr{u-\P}{\eta_{I,p'}}} 
	\leq \norm{u-\P \sep L_p(Q(I))} \cdot \norm{\eta_{I,p'} \sep L_{p'}(Q(I))}
	\lesssim \norm{u-\P \sep L_p(Q(I))}.
\end{align*}
Consequently, a Whitney-type argument (i.e., the application of \autoref{Prop:Whitney} stated in the Appendix with $t=\ell+\alpha$ and $q=\infty$) shows that
\begin{equation*}
	\abs{\distr{u}{\eta_{I,p'}}} 
	\lesssim \inf_{\P \in \Pi_\ell} \norm{u-\P \sep L_p(Q(I))}
	\lesssim \abs{Q(I)}^{(\ell+\alpha)/d+1/p} \abs{u}_{B^{\ell+\alpha}_\infty(L_\infty(Q(I))},
\end{equation*}
since we assumed $m>\ell$.
Next we use \link{def:QI} and estimate the Besov semi-norm by the H\"older semi-norm (see \autoref{Prop:semi_norms}) to obtain
\begin{align}
	\abs{\distr{u}{\eta_{I,p'}}} 
	&\lesssim 2^{-j(\ell+\alpha+d/p)} \abs{u}_{C^{\ell,\alpha}(Q(I))} \nonumber\\
	&\lesssim 2^{-j(\ell+\alpha+d/p)} \delta_{B(I)}^{-\gamma} \abs{u}_{C^{\ell,\alpha}_{\gamma,\loc}} \quad \; \text{for all} \; \quad 
	(I,\eta)\in\bigcup_{j\in\N_0} \Lambda^\mathrm{int}_j= \bigcup_{j\in\N_0}\bigcup_{n=C_0+1}^{C_1\,2^j} \Lambda_{j,n}, \label{est_int}
\end{align}
because the open cubes $Q(I)$ are contained in the closed balls $B(I)$ by definition.
For fixed $j\in\N_0$, $n\in\{C_0+1,C_0+2,\ldots,C_1 \,2^j\}$, and $(I,\eta)\in\Lambda_{j,n}$, we have
\begin{equation}\label{est_delta}
	\delta_{B(I)} 
	\geq \delta_{B_c(I)} 
	\geq \dist\!\left(2^{-j}k,\partial\Omega\right) - \frac{c\,\diam(Q)}{2}\,2^{-j}
	\geq (n-C_0)\,2^{-j}.
\end{equation}
Now let $\tau>0$ and recall the estimate $\abs{\Lambda_{j,n}}\lesssim 2^{j(d-1)}$ which we found in Step 1. 
Combining this with \link{est_int} and \link{est_delta} thus yields
\begin{align}
	\sum_{(I,\eta)\in\Lambda^\mathrm{int}_j} \abs{\distr{u}{\eta_{I,p'}}}^\tau
	&\lesssim \sum_{n=C_0+1}^{C_1\, 2^j} \sum_{(I,\eta)\in\Lambda_{j,n}} 2^{-j(\ell+\alpha+d/p)\tau} \, (n-C_0)^{-\gamma\tau} \,2^{j\gamma\tau} \abs{u}_{C^{\ell,\alpha}_{\gamma,\loc}}^\tau \nonumber\\
	&\lesssim \abs{u}_{C^{\ell,\alpha}_{\gamma,\loc}}^\tau \, 2^{-j(\ell+\alpha+d/p-\gamma)\tau+j(d-1)} \sum_{t=1}^{C_1\, 2^j} t^{-\gamma\tau},
	\qquad\qquad\qquad j\in\N_0. \label{est:int}
\end{align}
Note that, due to the assumption $\gamma > 0$, the quantity $\gamma\tau$ is always positive. Then straightforward calculations show that for all $j\in\N_0$
\begin{equation*}
	1 
	\leq \sum_{t=1}^{C_1\, 2^j} t^{-\gamma\tau}
	\lesssim \begin{cases}
		2^{j(1-\gamma\tau)} &\text{if} \quad \gamma\tau\in(0,1),\\
		1+j &\text{if} \quad \gamma\tau=1,\\ 
		1 &\text{if} \quad \gamma\tau>1,\\
	\end{cases}
\end{equation*}
such that we have to distinguish several cases for $\gamma$ in what follows:

\emph{Substep 4.1 (Small $\gamma$)}.
Let us consider the case $0<\gamma<(\ell+\alpha)/d+1/p$ first. Then obviously $d(\gamma - 1/p) < \ell + \alpha$, such that we can set
\begin{equation}\label{def:range_alpha}
	\tau =\left( \frac{\sigma}{d} + \frac{1}{p} \right)^{-1}
	\qquad \text{with} \qquad
	\max\!\left\{0, d \, \left( \gamma - \frac{1}{p} \right) \right\}
	< \sigma
	< \ell + \alpha.
\end{equation}
From $d(\gamma-1/p)<\sigma$ we particularly infer that $\gamma < \tau^{-1}$, i.e., $\gamma\tau < 1$, for this choice of $\tau$.
Therefore, from the considerations stated above we conclude that
\begin{align*}
	\sum_{j\in\N_0} \sum_{(I,\eta)\in\Lambda^\mathrm{int}_j} \abs{\distr{u}{\eta_{I,p'}}}^\tau
	&\lesssim \abs{u}_{C^{\ell,\alpha}_{\gamma,\loc}}^\tau \sum_{j\in\N_0} 2^{-j(\ell+\alpha+d/p-\gamma)\tau+j(d-1)+j(1-\gamma\tau)} \\
	&= \abs{u}_{C^{\ell,\alpha}_{\gamma,\loc}}^\tau \sum_{j\in\N_0} \left( 2^{d-(\ell+\alpha+d/p)\tau} \right)^j \\
	&\lesssim \abs{u}_{C^{\ell,\alpha}_{\gamma,\loc}}^\tau,	
\end{align*}
because the sum in the second line converges for $d-(\ell+\alpha+d/p)\tau<0$ which is equivalent to $\sigma<\ell+\alpha = \sigma^*$.
Similar to the end of Step 3, we note that the double sum on the left-hand side is equivalent to $\norm{u_2 \sep B^{\sigma}_\tau(L_\tau(\R^d))}^\tau$ such that \link{est:interior} follows (in the case of small $\gamma$) for all~$\sigma$ that satisfy \link{def:range_alpha}.
Note that if $\gamma > 1/p$, then the maximum in \link{def:range_alpha} is strictly positive. The result \link{est:interior} for $\sigma>0$ below this value can be deduced from the assertion we just proved by means of the standard embedding along the adaptivity scale:
\begin{equation*}
	B^{\sigma_2}_{\tau_2}(L_{\tau_2}(\R^d))
	\hookrightarrow B^{\sigma_1}_{\tau_1}(L_{\tau_1}(\R^d))
	\qquad \text{for all} \qquad \sigma_2 \geq \sigma_1 > 0,
\end{equation*}
where $1/\tau_i = \sigma_i/d+1/p$ for each $i\in\{1,2\}$.

\emph{Substep 4.2 (Large $\gamma$)}. We turn to the case
\begin{equation*}
	\frac{\ell+\alpha}{d} + \frac{1}{p} 
	\leq \gamma
	< \ell + \alpha + \frac{1}{p}.
\end{equation*}
As mentioned right after the statement of \autoref{theorem:embedding}, for $\gamma$ in this range we have that
\begin{equation*}
	\sigma^*
	= \frac{d}{d-1}\left(\ell+\alpha+ \frac{1}{p}-\gamma \right)
	\leq \ell+\alpha .
\end{equation*}
The lower bound for $\gamma$ thus implies that $\sigma^*\leq d\,\gamma-d/p$. Therefore, for every $0<\sigma<\sigma^*$ the corresponding $\tau$ in the adaptivity scale satisfies
\begin{equation*}
	\frac{1}{p} < \frac{1}{\tau} = \frac{\sigma}{d} + \frac{1}{p} < \gamma,
\end{equation*}
i.e., $\gamma\tau>1$. Hence, proceeding as in the previous substep yields
\begin{align*}
	\sum_{j\in\N_0} \sum_{(I,\eta)\in\Lambda^\mathrm{int}_j} \abs{\distr{u}{\eta_{I,p'}}}^\tau
	&\lesssim \abs{u}_{C^{\ell,\alpha}_{\gamma,\loc}}^\tau \sum_{j\in\N_0} 2^{-j(\ell+\alpha+d/p-\gamma)\tau+j(d-1)} \\
	&= \abs{u}_{C^{\ell,\alpha}_{\gamma,\loc}}^\tau \sum_{j\in\N_0} \left( 2^{d-1-\tau(\ell+\alpha+d/p-\gamma)} \right)^j
	\lesssim \abs{u}_{C^{\ell,\alpha}_{\gamma,\loc}}^\tau,	
\end{align*}
where this time the sum over $j$ converges if $d-1-\tau(\ell+\alpha+d/p-\gamma)<0$ which is (for the assumed range of $\gamma$) equivalent to $\sigma<\sigma^*$.
Since this implies the desired estimate \link{est:interior}, finally, the proof is complete.
\end{proof}

\begin{Rem}
The interested reader might ask what happens if $\gamma \geq \ell+\alpha+1/p$. For
$\gamma \geq \ell + \alpha + d/p$ the sum over \link{est:int} w.r.t.\ $j\in\N_0$ can never be convergent, because due to $\tau > 0$ the exponent $-j(\ell+\alpha+d/p-\gamma)\tau + j(d-1)$ would be non-negative for all $j$ and the sum over $t$ is bounded from below by $1$. Hence, we are left with $\ell+\alpha+1/p \leq \gamma < \ell+\alpha + d/p$. Choosing $\tau>0$ such that $\gamma \leq 1/\tau$ then implies $\sigma \geq d(\ell+\alpha)$ for $\sigma$ in the adaptivity scale. On the other hand, $\sigma<\ell+\alpha$ would be necessary for the geometric series to converge; see Substep 4.1.
In contrast, if we choose $\tau>0$ such that $\gamma > 1/\tau$, then convergence is equivalent to $\sigma<\frac{d}{d-1}(\ell+\alpha+1/p-\gamma)$ which contradicts $\sigma>0$ for the range of $\gamma$ under consideration.
\end{Rem}

\section{Besov regularity}
\label{sec:Besov_reg}
This section is concerned with the regularity of solutions to the $p$-Poisson equation \eqref{eq:p-Poisson}, $1<p<\infty$, in the adaptivity scale of Besov spaces $B^{\sigma}_{\tau}(L_{\tau}(\Omega))$, $1/ \tau = \sigma/d + 1/p$. 
In \autoref{subsec:p-Poisson} we deal with the general case of multidimensional, bounded Lipschitz domains.
The main result of this part, \autoref{thm:generic_besov}, describes (generic) sufficient conditions on the parameters of locally weighted H\"older spaces which ensure that the Besov regularity of all solutions $u$ to \eqref{eq:p-Poisson} that are contained in such spaces exceeds the Sobolev smoothness of $u$.
\autoref{subsec:p-Poisson_d=2} then is devoted to problems on two-dimensional domains, since there many more results concerning local H\"older regularity are available in the literature. 
Among other things, in this subsection, we state and prove explicit Besov regularity assertions for the unique solution to the $p$-Poisson equation \eqref{eq:p-Poisson}, with a right-hand in $L_q(\Omega)$, $q\geq p'$, which satisfies a homogeneous Dirichlet boundary condition. 
These statements constitute the main results of the present paper. 
In \autoref{thm:besov_reg_2d_lip_max} we deal with general bounded Lipschitz domains $\Omega\subset\R^2$, whereas \autoref{theorem:Besov_reg_2d} contains the results for the special case of bounded polygonal domains.

Existence and uniqueness of weak solutions to all problems we are going to consider is guaranteed by the following fairly general result which is well-known in the literature. Its proof can be found, e.g., in Lions \cite[Chapter 2]{Lio1969}.

\begin{Prop}[Existence and uniqueness] \label{Prop:existence}
	For $d\geq 2$ let $\Omega \subset \R^d$ denote a bounded domain and let $1 < p < \infty$. 
	Moreover, assume $f \in W^{-1}(L_{p'}(\Omega))$, as well as $g \in W^1(L_p(\Omega))$. 
	Then the problem
	\begin{align}
		\begin{split}
			-\div{\lvert \nabla u \rvert^{p-2} \nabla u} = f \;\;\quad \text{in} \quad \Omega, \\
			 u-g  \in W^1_0(L_p(\Omega)), \label{eq:nonhom_dirichlet_condition}
		\end{split}
	\end{align}
	admits a unique weak solution $u \in W^1(L_p(\Omega))$.
\end{Prop}

\begin{Rem}
Note that, since we like to deal with bounded Lipschitz domains $\Omega$ and $q \geq p'$, the chain of embeddings 
	\begin{equation*}
		L_q(\Omega)\hookrightarrow L_{p'}(\Omega)\hookrightarrow W^{-1}(L_{p'}(\Omega))
	\end{equation*}
	 together with \autoref{Prop:existence} (applied for $g\equiv 0$) guarantees that there is at least one $u\in W^{1}(L_p(\Omega))$ that solves the $p$-Poisson equation \link{eq:p-Poisson} with $f \in L_q(\Omega)$.
\end{Rem}

In order to prove non-trivial Besov regularity results, we will make use of the general embedding \autoref{theorem:embedding}. 
For that reason, we need to determine preferably small spaces $B^{s}_p(L_p(\Omega))$ and $C^{\ell,\alpha}_{\gamma, \loc}(\Omega)$ which still contain the solution $u$ to the respective problem under consideration. 
Clearly, smoothness results w.r.t.\ the Besov scale $B^{s}_p(L_p(\Omega))$ can be derived easily from corresponding Sobolev regularity assertions using the intimate relation of Sobolev and Besov spaces described in \autoref{rem:Besov_prop}(iv). 
The local H\"older regularity of solutions to the $p$-Poisson equation (\ref{eq:p-Poisson}), as well as to more general quasi-linear elliptic problems, was studied in several papers. 
We refer, e.g., to Ural'ceva \cite{Ura1968}, Uhlenbeck \cite{Uhl1977}, Evans \cite{Eva1982}, Lewis~\cite{Lew1983}, DiBenedetto~\cite{DiB1983}, Tolksdorf \cite{Tol1984}, Diening, Kaplick{\'y} and Schwarzacher \cite{DieKapSch2012}, Kuusi and Mingione \cite{KuuMin2014}, as well as to Teixeira \cite{Tei2014}.
The subsequent proposition can be derived as a special case from \cite[Corollary 5.5]{DieKapSch2012} (see also \cite[Remark 5.7]{DieKapSch2012}).

\begin{Prop}[$C^{1,\alpha}_{\loc}(\Omega)$ regularity] \label{Prop:hoelder_reg}
For $d\geq 2$ let $\Omega \subset \mathbb{R}^d$ denote any bounded domain, let $1 < p < \infty$, and $q > d$. Then there exists $\alpha \in (0,1)$ such that all $u \in W^1(L_p(\Omega))$ which are weak solutions to \eqref{eq:p-Poisson} with  $f \in L_q(\Omega)$ belong to $C^{\ell,\alpha}_{\loc}(\Omega)$.
\end{Prop}

\begin{Rem}\label{rem:sharp_hoelder}
It is well-known that, for $p>2$, solutions to (\ref{eq:p-Poisson}) do not possess continuous second derivatives in general, even if $f$ is smooth. For instance, a weak solution to the equation
	\begin{equation*}
	\div{\lvert \nabla u \rvert^{p-2} \nabla u} = 1 \quad \text{on} \quad \mathring{B}_1(0)
	\end{equation*}
	is given by
	\begin{equation*}
	u(x_1,\ldots,x_d)= \frac{p}{p-1}\abs{x_1}^{p/(p-1)}, 
	\end{equation*}
	see \cite[Proposition 5.4]{Sci2014} and \cite{LinLin2013}. Hence, in this respect $\ell=1$ in \autoref{Prop:hoelder_reg} is sharp at least for $p>2$.
\end{Rem}
Here and in what follows we shall say a given problem is of \emph{sharp regularity~$\alpha$} if $\alpha$ is a lower bound for the smoothness (measured in a certain scale) of \emph{all} solutions to \emph{any} problem instances (e.g., for all Lipschitz domains $\Omega$ and each $f\in L_{p'}(\Omega)$), but for every $\varepsilon>0$ there exists a problem instance such that its corresponding solution has a regularity strictly less than $\alpha^* + \varepsilon$.

\subsection{The $p$-Poisson equation in arbitrary dimensions} \label{subsec:p-Poisson}

Regularity results for partial differential equations are usually stated in terms of shift theorems. 
Concerning the $p$-Poisson equation \link{eq:p-Poisson} with homogeneous Dirichlet boundary conditions, 
\begin{align}
	\begin{split}
		-\div{\lvert \nabla u \rvert^{p-2} \nabla u} &= f 			\quad \text{in } \Omega, \\
		u &= 0 \quad \text{on } \partial \Omega, 		
		\label{eq:dirichlet_condition}
	\end{split}
\end{align}
and the scale of Sobolev spaces $W^{s}(L_{p}(\Omega))$ one such result is due to Savar{\'e} \cite[Theorems 2 and 2']{Sav1998}:

\begin{Prop}[Sobolev regularity on Lipschitz domains] \label{Prop:sobolev_reg_poisson_savare}
For $d\geq 2$ let $\Omega \subset \R^d$ be a bounded Lipschitz domain. 
Given $1<p<\infty$ and $f\in W^{-1}(L_{p'}(\Omega))$ let $u \in W_0^1(L_p(\Omega))$ denote the unique solution to \link{eq:dirichlet_condition}. 
Then, for $\theta\in[0,1)$,
\begin{equation}\label{cond:f}
	f \in W^{t_\theta}(L_{p'}(\Omega)) 
	\quad \text{with} \quad 
	t_\theta 
	= \begin{cases}
		-1+\theta/2 &\quad \text{if} \quad 1<p\leq 2,\\
		-1 + \theta/p' &\quad \text{if} \quad 2<p< \infty 
	\end{cases}
\end{equation}
implies that
\begin{equation*}
	u\in W^{s_\theta}(L_p(\Omega)) \quad \text{with} \quad 
	s_\theta 
	= \begin{cases} 
		1 +\theta/2  &\quad \text{if} \quad 1<p\leq 2,\\
		1 + \theta/p &\quad \text{if} \quad 2<p< \infty.
	\end{cases}
\end{equation*}
\end{Prop}

\begin{Rem}\label{rem:sharp}
In \cite[Remark 4.3]{Sav1998} Savar{\'e} states that the regularity results given in \autoref{Prop:sobolev_reg_poisson_savare} are sharp (in the sense defined above), even for the class of smooth domains.
\end{Rem}

Observe that $L_{p'}(\Omega) \hookrightarrow W^{t_\theta}(L_{p'}(\Omega))$ for all $\theta$ under consideration.
Hence, provided that $f \in L_{p'}(\Omega)$, the preceding \autoref{Prop:sobolev_reg_poisson_savare} shows that the unique solution to \link{eq:dirichlet_condition} is contained in $W^{s}(L_p(\Omega))$ and in $B_p^{s}(L_p(\Omega))$, respectively, for all $s<s^*$, where we set
\begin{equation}\label{eq:max_sobolev_reg}
		s^*= \begin{cases} 
			3/2 		\quad &\text{if} \quad 1 < p \leq 2, \\
			1 + 1/p		\quad &\text{if} \quad 2 < p <\infty.
		 \end{cases}
\end{equation}
Moreover, let us mention that Savar{\'e} actually proved (for an even larger class of equations and slightly weaker assumptions on $f$) that we may replace $B_p^{s}(L_p(\Omega))$, $s<s^*$, by $B^{s^*}_\infty(L_p(\Omega))$.
However, this slightly stronger assertion would not provide any gain in what follows.

\begin{Rem}\label{rem:sharp_sobolev}
In addition to \autoref{rem:sharp} we state that there are good reasons to assume that $s^*$ given in \link{eq:max_sobolev_reg} defines a sharp bound for the Sobolev regularity of solutions $u$ to \link{eq:dirichlet_condition}, even for much smoother right-hand sides $f$.
First of all, this conjecture is supported by the well-known fact that there exist Lipschitz domains $\Omega$ such that the solution for $p=2$ and some $f\in C^{\infty}(\overline{\Omega})$ does not belong to any $W^{3/2+\varepsilon}(L_2(\Omega))$, $\varepsilon>0$; see, e.g., Jerison and Kenig \cite[Theorem~A]{JerKen1995}. 
Moreover, for $d=2$ and $p>2$ it can be seen easily that $s^*=1+1/p$ can not be improved for general Lipschitz domains, as the following example shows: 
Given $\omega\in(0,2\pi)$ let
\begin{equation*}
	\mathfrak{C}(\omega)=\left\{ (r,\theta)\in [0,\infty) \times [0,2\pi] \sep  0<r<1 \text{ and } 0<\theta<\omega \right\}
\end{equation*}
denote an (open) circular sector of radius $1$ which is centered at the origin and possesses a central angle $\omega$.
Then, by \cite[Theorem 3]{Dob1989} (see also \cite{Aro1986}), there exist $\alpha(\omega)>0$ which can be computed explicitly and some function $t$ such that, under quite mild conditions on the right-hand side $f$, for every solution $u$ to \eqref{eq:dirichlet_condition} in $\Omega=\mathfrak{C}(\omega)$ there exist a positive constant $k$ and a function $v$ such that
\begin{align}
\label{eq:counter_1}
	u(r,\theta)=k \cdot r^{\alpha(\omega)} t(\theta) + v(r,\theta), \qquad (r,\theta)\in\mathfrak{C}(\omega),
\end{align} 
where $v$ fulfills
\begin{align}
\label{eq:counter_2}
	\abs{v(r,\theta)}\lesssim r^{\alpha(\omega)+\eta} \qquad \text{and} \qquad \abs{\nabla v(r,\theta)} \lesssim r^{\alpha(\omega)+\eta} 
\end{align}
for some absolute constant $\eta >0$. It follows from \eqref{eq:counter_1}, \eqref{eq:counter_2}, and the special structure of $t(\theta)$, cf.\ \cite[Theorem 1]{Dob1989}, that $\abs{\nabla u(r,\theta)} \sim r^{\alpha(\omega)-1}$ near the origin. 
Therefore $\abs{\nabla u} \in L_{\mu}(\mathfrak{C}(\omega))$ can hold true only if $\mu\cdot (\alpha(\omega)-1) > - 2$.
On the other hand, the behaviour of $\alpha(\omega)$ for large central angles $\omega$, is known: It has been shown that
\begin{align}
\label{eq:limalpha}
\lim_{\omega \rightarrow 2\pi} \alpha(\omega)=\frac{p-1}{p}.
\end{align}
Hence, by $\eqref{eq:limalpha}$, for every $\mu>2p$ there exists a two-dimensional Lipschitz domain $\Omega=\mathfrak{C}(\omega)$ and a solution $u$ to \eqref{eq:dirichlet_condition} such that $\abs{\nabla u}$ does not belong to $L_{\mu}(\Omega)$.
Consequently, for this solution Sobolev's embedding yields that $\abs{\nabla u}$ is not contained in $W^{1/p+\varepsilon}(L_p(\Omega))$ for any $\varepsilon>0$  and thus
\begin{align}\label{eq:u_notin_w}
	u \notin W^{1+1/p+\varepsilon}(L_p(\Omega)).
\end{align}
Finally, let us remark that for the open circular sector with $\omega=2\pi$ the same arguments yield \link{eq:u_notin_w} with $\varepsilon=0$. 
However, note that then $\Omega=\mathfrak{C}(2\pi)$ is not a Lipschitz domain anymore.
\end{Rem}

Unfortunately, if $d\geq 3$, then (to our best knowledge) finding the sharp local H\"older regularity $\alpha$ of solutions to \link{eq:p-Poisson}, \link{eq:nonhom_dirichlet_condition}, or \link{eq:dirichlet_condition}, respectively, still is an open problem. 
Moreover, in the articles mentioned before the statement of \autoref{Prop:hoelder_reg}, there appear too many unspecified constants that do not seem to allow estimates for the local H\"older semi-norms which are sufficient for our purposes, i.e., to obtain a satisfactory bound for the parameter~$\gamma$. In contrast, for the case $d=2$ much more explicit results are available such that these two drawbacks can be resolved. 
Consequently, we present a detailed discussion of the two-dimensional case in \autoref{subsec:p-Poisson_d=2}. 
To conclude the current subsection, at least we want to determine the \emph{range} of the parameters $\alpha$ and $\gamma$ for which the Besov regularity of the solution~$u$ (in the general multidimensional setting) \emph{would} exceed its Sobolev regularity.

\begin{theorem}\label{thm:generic_besov}
	For $d\geq 2$ let $\Omega \subset \mathbb{R}^d$ denote a bounded Lipschitz domain. 
	Moreover, for $1 < p < \infty$ and $f \in W^{-1}(L_{p'}(\Omega))$ let $u$ be a weak solution to \link{eq:p-Poisson} which satisfies $u \in W^s(L_p(\Omega))$ for all $s<\overline{s}\in[\ell,\ell+1)$ with some $\ell\in\N$.
	If, additionally, $u$ is contained in $C^{\ell,\alpha}_{\gamma, \loc}(\Omega)$ with
	\begin{align}\label{ineq:range_alpha_gamma}
		\overline{s}-\ell < \alpha \leq 1
		\qquad \text{and} \qquad 
		0 < \gamma < \ell + \alpha + \frac{1}{p} - \frac{d-1}{d} \, \overline{s},
	\end{align}
	then there exists $\overline{\sigma} > \overline{s}$ such that
	\begin{equation*}
		u \in B^{\sigma}_{\tau}(L_{\tau}(\Omega))
		\qquad \text{for all} \qquad 
		0< \sigma < \overline{\sigma}
		\qquad \text{and} \qquad 
		\frac{1}{\tau} = \frac{\sigma}{d} + \frac{1}{p}.
	\end{equation*}
\end{theorem}

Before proving \autoref{thm:generic_besov} we stress that, according to \autoref{Prop:existence}, we know that there indeed exists $\overline{s}\geq 1$ such that all solutions $u$ to the $p$-Poisson equation \link{eq:p-Poisson} are contained in $W^s(L_p(\Omega))$ for all $s<\overline{s}$. 
Moreover, at least when dealing with homogeneous boundary conditions (i.e., solutions of \link{eq:dirichlet_condition}), it is reasonable to assume that $s\in[\ell,\ell+1)$ and that $u\in C^{\ell,\alpha}_{\gamma, \loc}(\Omega)$ with $\ell=1$; see \autoref{rem:sharp_sobolev} and \autoref{rem:sharp_hoelder}, respectively.
Hence, \autoref{thm:generic_besov} particularly describes a wide range of sufficient conditions which ensure that the Besov regularity $\sigma$ (measured in the adaptivity scale w.r.t.\ $L_{p}(\Omega)$) of solutions $u$ to \link{eq:dirichlet_condition} on bounded Lipschitz domains is strictly larger than its maximal Sobolev regularity $\overline{s}$. 
Moreover, we note that the upper bound $\overline{\sigma}$ can be calculated (from $p$, the regularity parameters $\ell$, $\alpha$, and $\gamma$, as well as the dimension~$d$), as the following proof shows.

\begin{proof}[Proof (of \autoref{thm:generic_besov}).]
Since we assume that $u\in W^s(L_p(\Omega))$, $s<\overline{s}$, standard embeddings (cf.~\autoref{rem:Besov_prop}) imply that $u\in B_p^{s}(L_p(\Omega))$ for all $s\in(0,\overline{s})$.
Then, for general $0<\alpha \leq 1$ and $0<\gamma<\ell+\alpha+1/p$,
our embedding result (\autoref{theorem:embedding}) states that the additional assumption $u\in C^{\ell,\alpha}_{\gamma, \loc}(\Omega)$ yields $u \in B^{\sigma}_{\tau}(L_{\tau}(\Omega))$, $1/\tau=\sigma/d+1/p$, for all
\begin{align*}
	0 
	< \sigma
	< \min \left\{ \sigma^*, \frac{d}{d-1}\, (\overline{s}-\varepsilon) \right\}
	=: \overline{\sigma}, 
\end{align*}
where $\varepsilon>0$ can be chosen arbitrarily small and $\sigma^*$ depends on $d$, $p$, $\ell$, $\alpha$, and $\gamma$, as described in \eqref{def:max_alpha}. 
Thus, the maximal Besov regularity (w.r.t.\ the adaptivity scale) $\overline{\sigma}$ of the solution $u$ exceeds its maximal Sobolev regularity $\overline{s}$ provided that $\sigma^*>\overline{s}$.
Due to \eqref{def:max_alpha}, this is the case if $\alpha$ and $\gamma$ satisfy
\begin{align*}
	\ell + \alpha 
	> \overline{s}
	\quad \text{ and } \quad 
	0 
	< \gamma < \frac{\ell + \alpha}{d} + \frac{1}{p},
\end{align*}
or if
\begin{align}
	\frac{d}{d-1}\left( \ell + \alpha + \frac{1}{p} - \gamma \right) 
	> \overline{s} 
	\quad \text{ and } \quad 
	\frac{\ell + \alpha}{d} + \frac{1}{p} \leq \gamma 
	< \ell + \alpha + \frac{1}{p}. \label{eq:condition_to_parameters}
\end{align}
Now the first inequality in \eqref{eq:condition_to_parameters} is equivalent to $\gamma < \ell + \alpha + 1/p - \overline{s} \, (d-1)/d$ such that \eqref{eq:condition_to_parameters} reduces to
\begin{align*}
	\frac{\ell + \alpha}{d} + \frac{1}{p} 
	\leq \gamma < \ell + \alpha + \frac{1}{p} - \frac{d-1}{d} \, \overline{s}.
\end{align*}
This range for $\gamma$ is non-empty if and only if $\ell + \alpha > \overline{s}$. 
In summary, the condition $\ell + \alpha > \overline{s}$ is necessary in both cases and the union of the two ranges for $\gamma$ yields that $\sigma^* > \overline{s}$ for all values of $\alpha$ and $\gamma$ satisfying \link{ineq:range_alpha_gamma}, as claimed.
\end{proof}

\subsection{The $p$-Poisson equation in two dimensions} \label{subsec:p-Poisson_d=2}
As mentioned earlier, in order to derive non-trivial Besov regularity results by means of \autoref{theorem:embedding}, we need to determine (preferably small) spaces $C^{\ell,\alpha}_{\gamma, \loc}(\Omega)$ which contain the solutions $u$ to the $p$-Poisson equation \link{eq:p-Poisson}; see \autoref{subsec:Hoelder} for the definition of these spaces.  
For this purpose we proceed as follows. 
Starting from a known local H\"older regularity result, we estimate the H\"older semi-norms $\abs{u}_{C^{\ell,\alpha}(K)}$ on compact subsets $K \subset \subset \Omega$ in terms of $\delta_K$, in order to conclude estimates on the parameter $\gamma$.
In what follows we restrict ourselves to the situation $d=2$, because in this case explicit bounds on the (local) H\"older regularity are available in the literature.
In particular, quite recently Lindgren and Lindqvist \cite{LinLin2013} have proven a lower bound for the H\"older exponent of solutions to \link{eq:p-Poisson} with right-hand side $f \in L_q(\Omega)$, $q > 2$; see \autoref{Prop:hoelder_reg_plane} below.

\begin{Rem}
We note in passing that in dimension two we have $L_q(\Omega)\hookrightarrow W^{-1}(L_{p'}(\Omega))$, provided that $2/q < 1+ 2/p'$. Hence, \autoref{Prop:existence} guarantees that the problem \link{eq:p-Poisson} is uniquely solvable for all $1<p<\infty$ and $q>2$.
\end{Rem}

The subsequent definition is inspired by \cite{LinLin2013}.
\begin{defi} \label{def:opt_hoelder_reg}
Let us define the local H\"older exponent $\alpha^*_q=\alpha^*_q(p)$ for $2<q\leq \infty$ by
\begin{itemize}
	\item[$\bullet$)] $1<p\leq 2$:\quad\, If $q=\infty$, let $\alpha_q^*$ be any number less than $1$, and if $q<\infty$, let
	\begin{align*}
		\alpha_q^* = 1 - \frac{2}{q}.
	\end{align*}
	\item[$\bullet$)] $2<p<\infty$:\;\: If $q=\infty$, let $\alpha_q^*$ be any number less than $1/(p-1)$, and if $q<\infty$, let
	\begin{align*}
		\alpha_q^* = \frac{1- 2/q}{p-1}.
	\end{align*}
\end{itemize}
\end{defi}

The result of Lindgren and Lindqvist \cite[Theorem~3]{LinLin2013} then reads as follows.
\begin{Prop} \label{Prop:hoelder_reg_plane}
	Let $\Omega \subset \mathbb{R}^2$ be a bounded domain and let $1<p<\infty$. 
	For $2<q\leq \infty$, let $f\in L_q(\Omega)$ and set $\alpha=\alpha_q^*$ as specified in \autoref{def:opt_hoelder_reg}. Moreover, let $u \in W^1(L_p(\Omega))$ be a solution to \link{eq:p-Poisson}.
	Then $u \in C^{1,\alpha}_{\loc}(\Omega)$ and for any compact set $K \subset \Omega$, it holds
	\begin{align}\label{eq:hoelder_estimate_LiLi}
		\lvert u \rvert_{C^{1,\alpha}(K)} 
		\leq C(q,p,\alpha,K) \,\max\!\left\{ \norm{f \sep L_q(\Omega)}^{1/(p-1)}, \norm{u \sep L_{\infty}(\Omega)} \right\}. 
	\end{align}
\end{Prop}

\begin{Rem}\label{rem:opt_hoelder}
It is known that the H\"older exponent $\alpha_q^*$ defined above is sharp, at least for $p>2$ and $2<q \leq \infty$. If $q=\infty$, then this follows from the example given in \autoref{rem:sharp_hoelder}. Corresponding examples for finite $q$ can be found in \cite{LinLin2013}.
\end{Rem}

Based on the local H\"older regularity result given in \autoref{Prop:hoelder_reg_plane}, we are able to show that, for $\alpha=\alpha_q^*$ and certain values of $\gamma$, solutions to the $p$-Poisson equation \link{eq:p-Poisson} are contained in locally \emph{weighted} H\"older spaces $C^{1,\alpha}_{\gamma, \loc}(\Omega)$, too; see \autoref{Prop:hoelder_reg_plane2} below.
To do so, we have to examine the dependence of the constant $C(q,p,\alpha,K)$ in \eqref{eq:hoelder_estimate_LiLi} on $K\subset\subset \Omega$. This is performed in the subsequent lemma.

\begin{lemma}\label{lem:weighted_hoelder_est}
	Let the assumptions of \autoref{Prop:hoelder_reg_plane} be satisfied.
	Then, for every disc $B_{r/4} \subset \Omega$ of radius $r/4>0$ such that $\mathring{B}_{2r}$ is contained in $\Omega$ as well, we have
			\begin{align}	\label{eq:d2hoelder_first}
				\lvert u \rvert_{C^{1,\alpha}(B_{r/4})} 
				\leq C(q,p,\alpha, \Omega) \, r^{- \alpha -1} \, \max \!\left\{ \norm{f \sep L_q(B_{r})}^{1/(p-1)}, \norm{u \sep L_{\infty}(B_{r})} \right\}
			\end{align}
			and, for $t>2$,
			\begin{align}   \label{eq:d2hoelder_second}
				\lvert u \rvert_{C^{1,\alpha}(B_{r/4})} 
				\leq \hat{C}(q,p,\alpha, \Omega, t) \, r^{- \alpha -2/t} \, \max \!\left\{ \norm{f \sep L_q(B_{2r})}^{1/(p-1)}, \norm{ \nabla u  \sep L_t(B_{2r})} \right\}.
			\end{align}
\end{lemma}
\begin{proof}
To show the claim, assume that $u$ solves \link{eq:p-Poisson} on the whole domain $\Omega$ and let $B_r(x_0)\subset\Omega$ denote a disc of radius $r>0$ around an arbitrary point $x_0$. 
Then, certainly, $u$ is a solution of the restricted problem $\div{\lvert \nabla u \rvert^{p-2} \nabla u} = f$ in $B_{r}(x_0)$, as well. 
Moreover, from \autoref{Prop:hoelder_reg_plane} we infer that $u$ belongs to $C^{1,\alpha}_{\loc}(\Omega)$ with $\alpha=\alpha^*_q$ given in \autoref{def:opt_hoelder_reg}.
Hence, in particular $u \in L_{\infty}(B_{r}(x_0))$. 

Now let us perform a translation to the origin.
One checks easily that then $\tilde{u} = u(\cdot + x_0)$ solves 
\begin{equation*}
\div{\lvert \nabla \tilde{u} \rvert^{p-2} \nabla \tilde{u}} = \tilde{f} \quad \text{in} \quad B_r(0),
\end{equation*}
where $\tilde{f} = f(\cdot + x_0)$. 
Thus, it suffices to prove \eqref{eq:d2hoelder_first} and \eqref{eq:d2hoelder_second} only for solutions to the $p$-Poisson equation \link{eq:p-Poisson} in $B_r(0),\,r > 0$. 

To do so, we use a result for the unit disc $B_1(0)$. By \autoref{Prop:hoelder_reg_plane}, with $\overline{\Omega} = B_1(0)$ and $K = B_{1/4}(0)$, we know that if $u$ solves $\div{\lvert \nabla u \rvert^{p-2} \nabla u} = f$ in $B_1(0)$ with $u \in L_{\infty}(B_1(0))$ and $f \in L_q(B_1(0))$, then there exists a constant $C = C(q,p,\alpha)>0$, such that for all $x,y \in B_{1/4}(0)$ it holds
\begin{align}
	\lvert \nabla u(x) - \nabla u(y) \rvert 
	\leq C \lvert x - y \rvert^{\alpha} \,\max \!\left\{ \norm{f \sep L_q(B_1(0))}^{1/(p-1)}, \norm{u \sep L_{\infty}(B_1(0))} \right\}. \label{eq:hoelder_unit_disc}
\end{align}
Now suppose that $u$ solves $\div{\lvert \nabla u \rvert^{p-2} \nabla u} = f$ in some dilated disc $B_r(0)$ and 
let $F = r^p\,f(r\cdot)$. 
Then it is easy to see that $U = u(r\cdot)$ solves 
\begin{equation*}
\div{\lvert \nabla U \rvert^{p-2} \nabla U} = F \quad \text{in} \quad B_1(0).
\end{equation*}  
Clearly, $\norm{F \sep L_q(B_1(0))} = r^{p - 2/q} \norm{f \sep L_q(B_r(0))}$ and $\norm{U \sep L_{\infty}(B_1(0))} = \norm{u \sep L_{\infty}(B_r(0))}$. Next, we apply the estimate \eqref{eq:hoelder_unit_disc} to $U$ which yields that for all $x,y \in B_{1/4}(0)$
\begin{align*}
	&\abs{\nabla u(rx) - \nabla u(ry)} \\
	&\quad\qquad = r^{-1} \abs{\nabla U(x) - \nabla U(y)}  \\ 
	&\quad\qquad \leq C \, r^{-1} \abs{x - y}^{\alpha} \, \max\!\left\{ \norm{F \sep L_q(B_1(0))}^{1/(p-1)}, \norm{U \sep L_{\infty}(B_1(0))} \right\} \\
	&\quad\qquad \leq C \, r^{-1-\alpha} \abs{rx - ry}^{\alpha} \, \max\!\left\{ r^{(p-2/q)/(p-1)} \norm{f \sep L_q(B_r(0))}^{1/(p-1)}, \norm{u \sep L_{\infty}(B_r(0))} \right\}.
\end{align*}
Hence, for all $x \neq y$ in $B_{r/4}(0)$ it holds
\begin{align}
	&\frac{\abs{\nabla u(x) - \nabla u(y)}}{\abs{x - y}^{\alpha}} \nonumber\\
	&\qquad \qquad \leq C \, r^{-1-\alpha} \, \max\!\left\{ r^{(p-2/q)/(p-1)} \norm{f \sep L_q(B_r(0))}^{1/(p-1)}, \norm{u \sep L_{\infty}(B_r(0))} \right\} \label{eq:d2hoelder_first_intermediate} \\
 	&\qquad\qquad \leq \tilde{C} \, r^{-1-\alpha} \, \max\!\left\{ \norm{f \sep L_q(B_r(0))}^{1/(p-1)}, \norm{u \sep L_{\infty}(B_r(0))} \right\}, \notag
\end{align}
where $\tilde{C} = C \cdot \max\!\left\{ 1, \text{diam}(\Omega)^{(p-2/q)/(p-1)} \right\}$ and $(p-2/q)/(p-1)>0$, since $2/q<1<p$. This shows \eqref{eq:d2hoelder_first} for all discs $B_{r/4}(0)$ under consideration.

We are left with the proof of \eqref{eq:d2hoelder_second} for these discs.
Note that if $u$ solves \eqref{eq:p-Poisson}, so does $u-c$ for every constant $c$. 
Hence, from \eqref{eq:d2hoelder_first_intermediate} we infer
\begin{align}
	\frac{\abs{\nabla u(x) - \nabla u(y)}}{\abs{x - y}^{\alpha}} 
	&\leq C \, r^{-1-\alpha} \, \max\!\left\{ r^{(p-2/q)/(p-1)} \norm{f \sep L_q(B_r(0))}^{1/(p-1)}, \norm{u - c \sep L_{\infty}(B_r(0))} \right\}, \label{eq:d2hoelder_second_intermediate}
\end{align}
whenever $x \neq y$ belong to $B_{r/4}(0)$.
Next we apply Whitney's estimate (see \autoref{Prop:Whitney}) with $k=1$, $d=2$, $p=\infty$, and $q=t$. 
Thus, for every $t > d=2$ and every square $Q \subset \Omega$, there exist constants $c$ and $C'$, such that 
\begin{align}\label{eq:d2hoelder_second_intermediate2}
	\norm{u-c \sep L_{\infty}(Q)} 
	\leq C' \, \abs{ Q }^{1/2 - 1/t} \, \abs{u}_{W^1(L_t(Q))}. 
\end{align}
Let $Q_r$ denote the square in $\R^2$ with sides parallel to the coordinate axes and side length $2r$ that contains $B_r(0)$. 
Using the fact that $\lvert Q_r \rvert^{1/2 - 1/t} = (2r)^{1-2/t}$, from \eqref{eq:d2hoelder_second_intermediate2} we conclude 
\begin{align}
	\norm{u-c \sep L_{\infty}(B_r(0))} 
	&\leq C' \, \abs{Q_r}^{1/2 - 1/t} \, \abs{u}_{W^1(L_t(Q_r))} 
	\leq C'' \, r^{1 - 2/t} \, \norm{\nabla u \sep L_t(B_{2r}(0))} \label{eq:d2hoelder_second_intermediate3}
\end{align}
Now, \eqref{eq:d2hoelder_second_intermediate} and \eqref{eq:d2hoelder_second_intermediate3} together yield the upper bound
\begin{align*}
	&\frac{\abs{\nabla u(x) - \nabla u(y)}}{\abs{x - y}^{\alpha}}\\
	&\qquad \quad \leq C \, C''\, r^{-2/t-\alpha} \, \max\!\left\{ r^{-1 + 2/t + (p-2/q)/(p-1)} \norm{f \sep L_q(B_r(0))}^{1/(p-1)}, \norm{\nabla u \sep L_t(B_{2r}(0))} \right\}.
\end{align*}
Since, clearly,
\begin{align*}
	-1 + \frac{2}{t} + \frac{p-2/q}{p-1}=\frac{2}{t}+\frac{1-2/q}{p-1} > 0,
\end{align*}
by setting $\hat{C} = C \cdot C''\cdot \max\!\left\{ 1, \diam(\Omega)^{2/t + (1-2/q)/(p-1)} \right\}$ we finally arrive at
\begin{align*}
	\frac{\abs{\nabla u(x) - \nabla u(y)}}{\abs{x - y}^{\alpha}}
	\leq \hat{C} \, r^{-2/t -\alpha} \, \max\!\left\{ \norm{f \sep L_q(B_{2r}(0))}^{1/(p-1)}, \norm{ \nabla u \sep L_t(B_{2r}(0))} \right\}
\end{align*}
for all $x \neq y$ in $B_{r/4}(0)$. 
This shows \eqref{eq:d2hoelder_second} for all discs of interest.
\end{proof}

The locally weighted H\"older regularity result which forms the basis for our further analysis now can be derived easily from \eqref{eq:d2hoelder_second}:
\begin{Prop}[$C^{1,\alpha}_{\gamma,\loc}(\Omega)$ regularity] \label{Prop:hoelder_reg_plane2}
Let $\Omega \subset \mathbb{R}^2$ be a bounded Lipschitz domain and assume $1<p<\infty$. 
	Furthermore, for $2<q\leq \infty$ and $f\in L_q(\Omega)$, let $u \in W^1(L_p(\Omega))$ be some solution to the $p$-Poisson equation \eqref{eq:p-Poisson} and set $\alpha=\alpha^*_q$ as in \autoref{def:opt_hoelder_reg}.
	\begin{itemize}
		\item[(i)] If $\abs{\nabla u} \in L_t(\Omega)$ for some $t>2$, then we have
			\begin{align} \label{eq:d2hoelder_third}
				u \in C^{1,\alpha}_{\gamma, \loc}(\Omega)
				\qquad \text{for} \qquad 
				\alpha=\alpha^*_q,
			\end{align}
			as well as every weight parameter $\gamma \geq \alpha + 2/t$.
		\item[(ii)] If $u\in W^s(L_p(\Omega))$ for all $s<\overline{s}$ with some $\overline{s}>\max\{2/p,1\}$, then \link{eq:d2hoelder_third} holds true for all
		\begin{equation*}
			\gamma >\alpha + \max\!\left\{0,1-\overline{s}+\frac{2}{p}\right\}.
		\end{equation*}
	\end{itemize} 
\end{Prop}

\begin{proof}
Let us prove (i). Since the locally weighted H\"older spaces $C^{1,\alpha}_{\gamma, \loc}(\Omega)=C^{1,\alpha}_{\gamma, \loc}(\Omega; \K(c))$ are monotone in $\gamma$ (see \autoref{Rem:locally_weighted_hoelder_spaces}), we may restrict ourselves to the limiting case $\gamma=\alpha + 2/t$.
Moreover, without loss of generality, we can assume $c > 8$; cf.\ \autoref{sec:spaces}. 
Then let us consider a compact disc $B_{r}\in \K(c)$, i.e., $B_{r}=B_{r}(x_0)$ with $x_0\in\Omega$ and $r>0$ such that the (open) disc $\mathring{B}_{c\, r}(x_0)$ still is contained in $\Omega$.
Clearly, $r < \dist(x_0,\partial\Omega)/8$, so that we can choose $R\geq r$ with
\begin{equation*}
	\frac{\dist(x_0,\partial\Omega)}{16} < R < \frac{\dist(x_0,\partial\Omega)}{8}.
\end{equation*}
Consequently, $B_R=B_R(x_0)$ is a compact disc with $B_r\subseteq B_R\subset \mathring{B}_{8R} \subset\Omega$.
Therefore, \eqref{eq:d2hoelder_second} applied for $B_R$ yields
\begin{align*}
	\lvert u \rvert_{C^{1,\alpha}(B_{r})} 
	\leq \lvert u \rvert_{C^{1,\alpha}(B_{R})}
	\leq C \, R^{- \alpha -2/t} \, \max \left\{ \norm{f \sep L_q(B_{8R})}^{1/(p-1)}, \norm{ \nabla u \sep L_{t}(B_{8R})} \right\},
\end{align*}
where $C = C(q,p,\alpha,\Omega)$ does not depend on $r$. 
Since $\delta_{B_r} < \dist(x_0,\partial\Omega)<16\,R$ and $\gamma = \alpha + 2/t$, setting $C'=C\cdot 16^\gamma$ we may estimate further
\begin{align*}
	\lvert u \rvert_{C^{1,\alpha}(B_{r})} 
	&\leq C' \, \delta_{B_r}^{- \gamma} \max \left\{ \norm{ f \sep L_q(\Omega)}^{1/(p-1)}, \norm{ \nabla u \sep L_t(\Omega)} \right\}.
\end{align*}
Observe that the latter maximum is finite due to the additional assumption that $\abs{\nabla u}$ belongs to $L_t(\Omega)$.
Multiplying by $\delta_{B_r}^{\gamma}$ and taking the supremum over all $B_{r}\in \K(c)$ thus proves the claim stated in (i).

The proof of (ii) follows from Sobolev's embedding: At first, note that $\overline{s}>2/p$ yields that $1>\max\{0,1-\overline{s}+2/p\}$. 
Therefore, we can choose $s<\overline{s}$ and $t>2$ such that $2/t>\max\{0,1-s+2/p\}$ is arbitrary close to $\max\{0,1-\overline{s}+2/p\}$. 
Thus, in view of \link{eq:d2hoelder_third}, it remains to show that $\abs{\nabla u}\in L_t(\Omega)$ for this choice of $s$ and $t$.
To do so, observe that $s-1>2/p-2/t$.
Since we imposed the additional condition that $\overline{s}>1$, we may assume that $s-1>0$. 
Hence, it follows
\begin{equation*}
	s-1 > 2 \cdot \max\!\left\{0, \frac{1}{p} - \frac{1}{t} \right\}
\end{equation*}
which particularly implies the embedding $W^{s-1}(L_p(\Omega)) \hookrightarrow L_t(\Omega)$.
Finally, the fact that $u\in W^s(L_p(\Omega))$ yields $\abs{\nabla u} \in W^{s-1}(L_p(\Omega))$ completes the proof.
\end{proof}

Next let us combine the locally weighted H\"older regularity result obtained in \autoref{Prop:hoelder_reg_plane2} above with the generic Besov regularity result stated in \autoref{thm:generic_besov}.
This leads to conditions on the Sobolev smoothness of solutions $u$ to the $p$-Poisson equation \link{eq:p-Poisson} which imply (non-trivial) Besov regularity assertions for these $u$.

\begin{theorem}\label{thm:besov_reg_2d_lip}
	Let $\Omega \subset \mathbb{R}^2$ be a bounded Lipschitz domain and assume $1<p<\infty$. Moreover, for $2<q\leq \infty$, as well as $f\in L_q(\Omega)$, let $u$ be some solution to the $p$-Poisson equation \link{eq:p-Poisson} which satisfies $u\in W^s(L_p(\Omega))$ for all $s<\overline{s}$.
	Then the conditions
	\begin{itemize}
		\item[$\bullet$)] $1<p \leq 2$\quad\, and \quad $\frac{2}{p} < \overline{s} < 2 - \frac{2}{q}$,
		\item[$\bullet$)] $2<p<\infty$\;\; and \quad $1 < \overline{s} < 1 + \frac{1-2/q}{p-1}$
	\end{itemize}
	imply that there exists $\overline{\sigma} > \overline{s}$ such that
	\begin{equation}\label{eq:u_in_besov}
		u \in B^{\sigma}_{\tau}(L_{\tau}(\Omega))
		\qquad \text{for all} \qquad 
		0< \sigma < \overline{\sigma}
		\qquad \text{and} \qquad 
		\frac{1}{\tau} = \frac{\sigma}{2} + \frac{1}{p}.
	\end{equation}
\end{theorem}

\begin{proof}
Note that our assumptions particularly imply
\begin{equation}\label{ineq:range_s}
	\max\!\left\{1, \frac{2}{p}\right\} < \overline{s} < 2.
\end{equation}
Therefore, in view of \autoref{thm:generic_besov} (applied with $d=2$ and $\ell=1$), it suffices to find parameters $\alpha$ and $\gamma$ with $\overline{s}-1 < \alpha \leq 1$ and
\begin{align}\label{ineq:range_gamma_2d}
	0 < \gamma < 1 + \alpha + \frac{1}{p} - \frac{\overline{s}}{2} 
\end{align}
such that $u\in C^{1,\alpha}_{\gamma, \loc}(\Omega)$.
Observe that from \link{ineq:range_s} it follows
\begin{equation*}
	\alpha + \max\!\left\{0,1-\overline{s}+\frac{2}{p}\right\} < 1 + \alpha + \frac{1}{p} - \frac{\overline{s}}{2}
	\qquad \text{for all} \qquad
	0<\alpha\leq 1.
\end{equation*}
Thus, due to \autoref{Prop:hoelder_reg_plane2}(ii), choosing $\alpha=\alpha_q^*$ (as given in \autoref{def:opt_hoelder_reg}), there exists~$\gamma$ which satisfies \link{ineq:range_gamma_2d} such that $u\in C^{1,\alpha}_{\gamma, \loc}(\Omega)$.
To complete the proof, it remains to check that this choice of $\alpha$ belongs to the interval $(\overline{s}-1,1]$ which is obvious in view of \autoref{def:opt_hoelder_reg}, as well as our restrictions on $\overline{s}$.
\end{proof}

\begin{Rem}
Note that the bound $\overline{\sigma}$ in \autoref{thm:besov_reg_2d_lip} can be calculated explicitly, provided that the maximal Sobolev regularity $\overline{s}$ is known; see, e.g., the proof of \autoref{thm:besov_reg_2d_lip_max} below.
\end{Rem}

Now we are well-prepared to state and prove one of the main results of this paper. 
It shows that for a large range of parameters $p$ and $q$ the (unique) solution to \link{eq:dirichlet_condition}, i.e., to the $p$-Poisson with homogeneous Dirichlet boundary conditions, has a significantly higher Besov regularity compared to its Sobolev smoothness.
Indeed, as we shall see, on bounded Lipschitz domains $\Omega\subset\R^2$ this happens whenever $4/3<p<\infty$ and $\max\{4, 2\, p\} < q \leq \infty$.
Therefore, for the same range of parameters, the application of adaptive (wavelet) algorithms for the numerical treatment of \link{eq:dirichlet_condition} is completely justified.
Recall that from \autoref{Prop:sobolev_reg_poisson_savare} (and the subsequent remarks) it follows that the solution $u$ to  this problem is contained in $W^s(L_p(\Omega))$ for all $s<s^*$ given in \link{eq:max_sobolev_reg}.
Consequently, the proof of the subsequent result is obtained by applying \autoref{thm:besov_reg_2d_lip} with $\overline{s}=s^*$ together with some straightforward calculations.

\begin{theorem}[Besov regularity on Lipschitz domains in $2$D]\label{thm:besov_reg_2d_lip_max}
	Let $\Omega \subset \mathbb{R}^2$ be a bounded Lipschitz domain, $1 < p < \infty$, as well as $f \in L_q(\Omega)$ with $2 < q  \leq \infty$ and $q \geq p'$.
	Then the unique solution $u$ to the $p$-Poisson equation with homogeneous Dirichlet boundary conditions \eqref{eq:dirichlet_condition} satisfies 
	\begin{equation*}
		u \in B^{\sigma}_{\tau}(L_{\tau}(\Omega))
		\qquad \text{for all} \qquad 
		0< \sigma < \overline{\sigma}
		\qquad \text{and} \qquad 
		\frac{1}{\tau} = \frac{\sigma}{2} + \frac{1}{p},
	\end{equation*}	
	where
	\begin{align*}
		\overline{\sigma} 
		= \begin{cases} 
			\frac{3}{2} \quad & \text{if} \quad 1<p < 4/3 \text{ and } p' \leq q \leq \infty, \\
			\frac{3}{2} \quad & \text{if} \quad p = 4/3 \text{ and } 4 < q \leq \infty, \\
			3- \frac{2}{p} \quad & \text{if} \quad 4/3 < p \leq 2 \text{ and } (\frac{1}{p}-\frac{1}{2})^{-1} \leq q \leq \infty,\\
			2 - \frac{2}{q} \quad & \text{if} \quad 4/3 < p \leq 2 \text{ and } 4 < q < (\frac{1}{p}-\frac{1}{2})^{-1},\\
			\frac{3}{2} \quad & \text{if} \quad 4/3 \leq p < 2 \text{ and } p' \leq q \leq 4,\\
			\frac{3}{2}  \quad & \text{if} \quad p=2 \text{ and } 2 < q \leq 4,\\
			1 + \frac{1- 2/q}{p-1}  \quad & \text{if} \quad 2 < p < \infty \text{ and } 2\,p < q \leq \infty,\\
			1 + \frac{1}{p}  \quad & \text{if} \quad 2 < p < \infty  \text{ and } 2<q \leq 2\,p.
		\end{cases}
	\end{align*}
\end{theorem}

\begin{proof}
\emph{Step 1.}
Let us start with the cases where $\overline{\sigma}=s^*$, i.e., where $\overline{\sigma}$ equals $3/2$ or $1+1/p$. Then from classical embeddings of Besov spaces it follows that $u\in W^s(L_p(\Omega))$ for all $0<s<\overline{s}$ implies that $u$ also belongs to $B^{s}_p(L_p(\Omega))$ for all these $s$ which in turn yields the claim; cf.\ \autoref{rem:Besov_prop}.

\emph{Step 2.}
We are left with proving the assertion for the third, fourth, and seventh line in the definition of $\overline\sigma$.
According to (the proof of) \autoref{thm:besov_reg_2d_lip} we know that in all these remaining cases \autoref{Prop:hoelder_reg_plane2}(ii) ensures the existence of some reasonably small $\gamma$ such that $u\in C^{\ell,\alpha}_{\gamma, \loc}(\Omega)$, where $\alpha=\alpha_q^*$ (as given in \autoref{def:opt_hoelder_reg}) and $\ell=1$.
In fact, it can be checked that we can use
\begin{equation*}
	\gamma 
	= \alpha + \varepsilon + \begin{cases}
		2/p-1/2 \quad & \text{if} \quad p<2,\\
		1/p,  	\quad & \text{if} \quad p\geq 2
	\end{cases}
\end{equation*}
with arbitrarily small $\varepsilon>0$. 
As shown in the proof of \autoref{thm:generic_besov} (which we used to derive \autoref{thm:besov_reg_2d_lip}), the desired quantity $\overline{\sigma}$ then is given by $\sigma^*$ defined in \link{def:max_alpha} in \autoref{theorem:embedding}. 
Thus, we need to determine whether our choice of $\gamma$ is smaller or larger than $(1+\alpha)/2+1/p$. 
Note that, according to \autoref{thm:generic_besov}, we already know that for all cases of interest it is smaller than $1+\alpha+1/p$.
It turns out that for $4/3<p \leq 2$ and $(1/p-1/2)^{-1} \leq q \leq \infty$, i.e., for the constellation described in the third line, the second case in \link{def:max_alpha} applies, i.e., then 
\begin{equation*}
	\frac{1+\alpha}{2}+\frac{1}{p} \leq \gamma < 1 + \alpha +\frac{1}{p}.
\end{equation*}
Consequently, for these $p$ and $q$, the quantity $\overline{\sigma}=\sigma^*$ is given by $2(1+\alpha+1/p-\gamma)=3-2/p-\varepsilon$, where $\varepsilon$ can be neglected since it can be chosen arbitrarily small.

For the remaining two ranges for $p$ and $q$ the chosen weight $\gamma$ is small enough such that the first case in \link{def:max_alpha} applies. 
Thus, for $p$ and $q$ as described in the fourth and seventh line, we obtain $\overline{\sigma}=\sigma^*=\ell+\alpha$ with $\ell=1$ and $\alpha=\alpha_q^*$.
This finishes the proof.
\end{proof}

In the more restrictive (but practically more important) setting of polygonal domains slightly better Besov regularity assertions for the unique solutions to \eqref{eq:dirichlet_condition} with $f\in L_q(\Omega)$ can be deduced using our method, at least for some cases.
For this purpose, we will employ a further Sobolev regularity result which was shown by Ebmeyer \cite[Corollary~2.3]{Ebm2002} for polyhedral Lipschitz domains in arbitrary dimensions:

\begin{Prop} \label{Prop:sobolev_reg_poisson_Ebmeyer_short}
For $d\geq 2$ let $\Omega \subset \mathbb{R}^d$ be a bounded polyhedral Lipschitz domain and for $1 < p < \infty$ let $f \in L_{p'}(\Omega)$. Then the unique solution $u \in W^1(L_p(\Omega))$ to (\ref{eq:dirichlet_condition}) satisfies
\begin{align*}
	\abs{\nabla u} \in L_t(\Omega)
\qquad \text{for all} \qquad t < \frac{d}{d-1} \, p.
\end{align*}
\end{Prop}

\begin{Rem}
The example described in \autoref{rem:sharp_sobolev} shows that, for $d=2$, Ebmeyer's result (\autoref{Prop:sobolev_reg_poisson_Ebmeyer_short}) is sharp, meaning that there are cases in which 
\begin{align*}
	\abs{\nabla u} \notin L_t(\Omega)
\qquad \text{if} \qquad t > 2p=\frac{d}{d-1} \, p.
\end{align*}
\end{Rem}

Our improved Besov regularity result for solutions to $p$-Poisson equations with homogeneous boundary conditions \link{eq:dirichlet_condition} on bounded polygonal domains then reads as follows.

\begin{theorem}[Besov regularity on polygonal domains] \label{theorem:Besov_reg_2d}
	Let $\Omega \subset \mathbb{R}^2$ denote a bounded polygonal domain and let $1 < p < \infty$, as well as $f \in L_q(\Omega)$ with $2 < q  \leq \infty$ and $q \geq p'$.
		Then the unique solution $u$ to the $p$-Poisson equation with homogeneous Dirichlet boundary conditions \eqref{eq:dirichlet_condition} satisfies 
	\begin{equation*}
		u \in B^{\sigma}_{\tau}(L_{\tau}(\Omega))
		\qquad \text{for all} \qquad 
		0< \sigma < \overline{\sigma}
		\qquad \text{and} \qquad 
		\frac{1}{\tau} = \frac{\sigma}{2} + \frac{1}{p},
	\end{equation*}	
	where
	\begin{align*}
		\overline{\sigma} 
		= \begin{cases} 
			2 - \frac{2}{q} \quad & \text{if} \quad 1<p < 4/3 \text{ and } p' \leq q \leq \infty,\\
			2 - \frac{2}{q} \quad & \text{if} \quad p=4/3 \text{ and } 4 < q \leq \infty,\\
			2 - \frac{2}{q} \quad & \text{if} \quad 4/3 < p \leq 2 \text{ and } 4 < q \leq \infty,\\
			\frac{3}{2} \quad & \text{if} \quad 4/3 \leq p < 2 \text{ and } p' \leq q \leq 4,\\
			\frac{3}{2}  \quad & \text{if} \quad p=2 \text{ and } 2 < q \leq 4,\\
			1 + \frac{1- 2/q}{p-1}  \quad & \text{if} \quad 2 < p < \infty \text{ and } 2\,p < q \leq \infty,\\
			1 + \frac{1}{p}  \quad & \text{if} \quad 2 < p < \infty  \text{ and } 2<q \leq 2\,p.
		\end{cases}
	\end{align*}
\end{theorem}

Before giving the proof of this assertion we want to stress that in the first three cases, as well as in the sixth one, the upper bound $\overline{\sigma}$ for the regularity of the solution $u$ in the adaptivity scale of Besov spaces is strictly larger than $\overline{s}=s^*$ as defined in \link{eq:max_sobolev_reg} which is considered to be a sharp bound for the regularity in the Sobolev scale; see \autoref{rem:sharp_sobolev}.
Hence, in contrast to \autoref{thm:besov_reg_2d_lip_max} (which deals with general bounded Lipschitz domains in $\R^2$), on polygonal domains $u$ gains some additional regularity also in the range $1<p\leq 4/3$ (except for the case $p=4/3$ and $q=4$). 
Furthermore, observe that for the case of $p\in(4/3,2)$ and large $q$ the value $3-2/p$ for Lipschitz domains is strictly worse than $2-2/q$ obtained in \autoref{theorem:Besov_reg_2d} for polygonal domains.
Finally we note that, given some fixed $p$, in all cases in which $\overline{\sigma}>\overline{s}$ this quantity grows  with increasing integrability $q$ of the right-hand side $f$. This is not the case for $s^*$.
Accordingly, the largest gain $\overline{\sigma}-\overline{s}$ is obtained for $f\in L_\infty(\Omega)$. 
This situation is illustrated in \autoref{fig:Besov_Sobolev_reg} below.
\begin{figure}[htb]
	\begin{center}
	\scalebox{.645}{\input{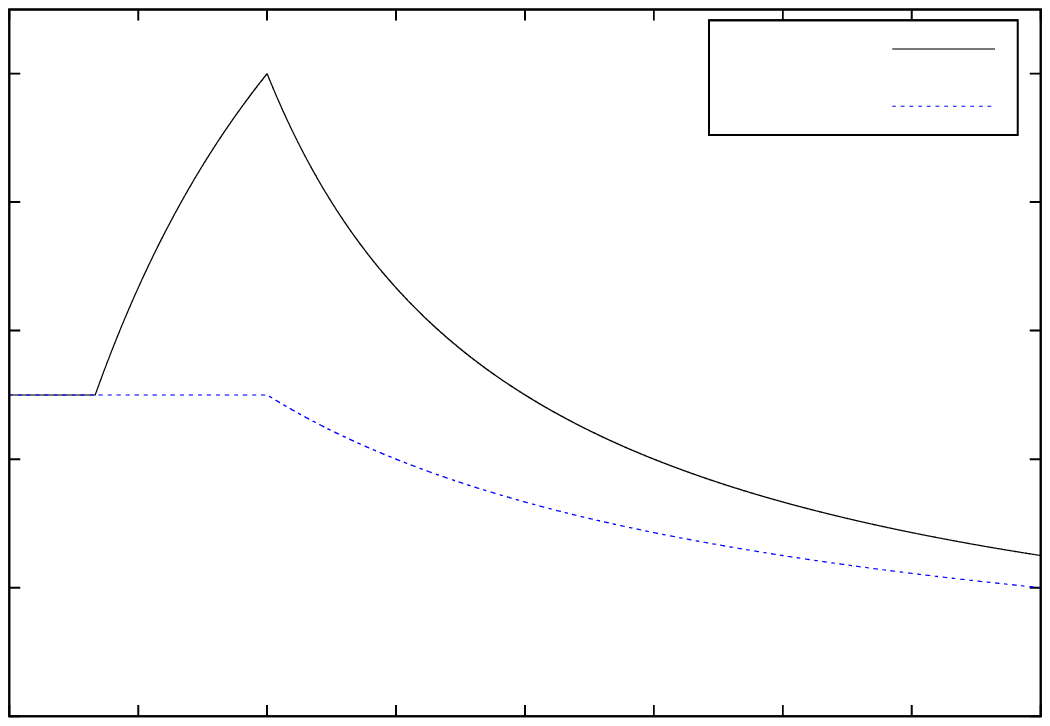}}
	\scalebox{.645}{\input{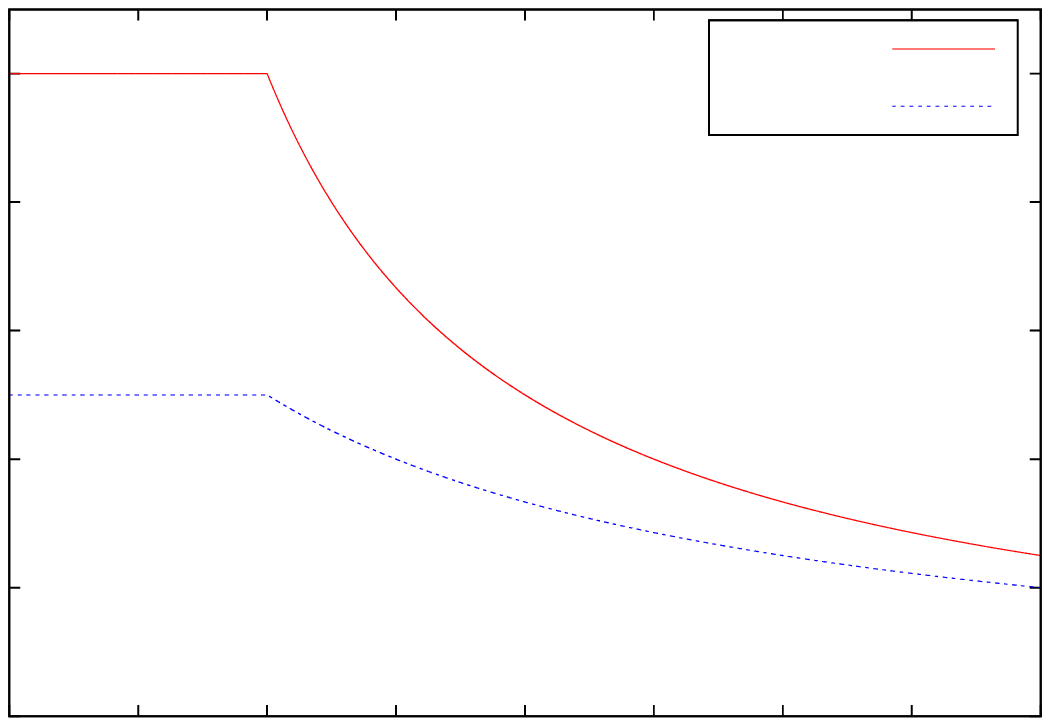}}
	\caption{Bounds $\overline{\sigma}$ and $s^*$ for the regularity of solutions $u$ to  \link{eq:dirichlet_condition} with $f \in L_{\infty}(\Omega)$ on bounded $2$D Lipschitz domains (left) and bounded polygonal domains (right), measured in $B^{\sigma}_{\tau}(L_{\tau}(\Omega))$, $1/ \tau = \sigma / 2 + 1/p$, and in $W^s(L_p(\Omega))$, respectively.}
	\label{fig:Besov_Sobolev_reg}
	\end{center}
\end{figure}

\begin{proof}[Proof (of \autoref{theorem:Besov_reg_2d})]
\emph{Step 1.}
Since $q\geq p'$, we have that $L_q(\Omega)\hookrightarrow L_{p'}(\Omega)\hookrightarrow W^{-1}(L_{p'}(\Omega))$.
Consequently, \autoref{Prop:existence} assures a unique solution $u\in W^{1}(L_p(\Omega))$. 
Then \autoref{rem:Besov_prop}(iv) implies $u\in B^{1-\varepsilon}_{p}(L_p(\Omega))$ for all $\varepsilon\in(0,1)$. 
Moreover, by \autoref{Prop:hoelder_reg_plane2}(i) we know that $u \in  C^{1,\alpha}_{\gamma, \loc}(\Omega)$ for all $\gamma \geq \alpha + 2/t$, with $\alpha=\alpha_q^*$ given in \autoref{def:opt_hoelder_reg} and $t>2$ such that $\abs{\nabla u}\in L_t(\Omega)$. 
\autoref{Prop:sobolev_reg_poisson_Ebmeyer_short} shows that the latter condition is fulfilled for all $t<2p$, i.e., for all $2/t$ strictly larger (but arbitrary close to) $1/p$.
Thus, since $\alpha \in (0,1)$, we can choose $\gamma$ such that
\begin{align*}
	\alpha + \frac{1}{p}<\gamma < \frac{1+\alpha}{2}+\frac{1}{p}.
\end{align*} 
Then, for this choice of $\alpha$ and $\gamma$, as well as $d=2$, $s=1-\varepsilon$, and $\ell=1$, we apply \autoref{theorem:embedding} (note that every polygonal domain $\Omega\subset\R^2$ is Lipschitz!) and conclude that
$u$ belongs to $ B^{\sigma}_{\tau}(L_{\tau}(\Omega))$, $1/\tau=\sigma/2+1/p$, for all 
\begin{align*}
	0 
	< \sigma 
	< \min\!\left\{ 1+ \alpha, \frac{2}{2-1} \, (1-\varepsilon)\right\}  	
	= 1 + \alpha,
\end{align*}
where the last equality holds provided that $\varepsilon>0$ is chosen sufficiently small.

\emph{Step 2.}
Since $f\in L_{p'}(\Omega)$, we furthermore can employ \autoref{Prop:sobolev_reg_poisson_savare} (as well as the subsequent remarks) to see that $u\in W^s(L_p(\Omega))$ for all $s<s^*$. This implies that $u$ belongs to $B^{s}_p(L_p(\Omega))$ and $B^{\sigma}_\tau(L_\tau(\Omega))$ for all $s$ and $\sigma$ less than $s^*$, respectively.

In conclusion, combining both steps yields
\begin{equation*}
	u \in B^{\sigma}_\tau(L_\tau(\Omega))
	\qquad \text{for all} \qquad 
	0<\sigma < \max\{1+\alpha, s^*\}
	\qquad \text{and} \qquad 
	\frac{1}{\tau} = \frac{\sigma}{2} + \frac{1}{p}.
\end{equation*}
Now the claim directly follows from the definitions of $\alpha=\alpha_q^*$ and $s^*$.
\end{proof}

\begin{Rem}
We add some comments on our main results in \autoref{thm:besov_reg_2d_lip_max} and \ref{theorem:Besov_reg_2d}, resp.:
\begin{itemize}
	\item[(i)] The restriction $q \geq p'$ in \autoref{thm:besov_reg_2d_lip_max} can be weakened. Anyhow, note that for $p$ in the vicinity of $1$ and $q$ close to $2$, \autoref{Prop:sobolev_reg_poisson_savare} only guarantees that the unique solution $u$ to \link{eq:dirichlet_condition} satisfies $u\in W^{s}(L_p(\Omega))$ for all $s<\overline{s}$ with some $1\leq \overline{s}<s^*$.
	\item[(ii)] According to \cite[Section~5.3]{Ebm2002} \autoref{Prop:sobolev_reg_poisson_Ebmeyer_short} remains valid for special classes of bounded Lipschitz domains with polyhedral structure. Hence, also \autoref{theorem:Besov_reg_2d} applies to this slightly generalized situation.
	
	\item[(iii)] Observe that for large $q$ our bound $\overline{\sigma}$ in \autoref{theorem:Besov_reg_2d} always equals $1+\alpha$, where $\alpha=\alpha^*_q$ is the local H\"older exponent given in \autoref{def:opt_hoelder_reg} which is known to be optimal at least for $p>2$; see \autoref{rem:opt_hoelder}. 
Thus, by \link{def:max_alpha}, as well as the subsequent statements, we see that the results stated in \autoref{theorem:Besov_reg_2d} are the best possible we can achieve by our method (i.e., by \autoref{theorem:embedding}). 
On the other hand, we do not know whether they are sharp, as (for general $p$) in the current literature there seem to exist no results at all which address comparable regularity questions.
However, for example in the case of the classical Laplacian ($p=2$) Besov regularity larger than two cannot be expected for general right-hand sides of smoothness zero, since then we deal with a linear operator of order two.
\end{itemize}
\end{Rem}

Finally, let us briefly consider \emph{$p$-harmonic} functions, i.e., solutions to the \emph{$p$-Laplace equation}
\begin{align}
	\div{\abs{\nabla u}^{p-2} \nabla u} 
	= 0 \qquad \text{in} \quad \Omega, \label{eq:p-laplace} 
\end{align}
where $\Omega \subset \mathbb{R}^2$ is a bounded domain and $1 < p < \infty$.
In \cite[Remark 2.5(iv)]{Ebm2002} Ebmeyer states that if $\Omega$ is a bounded polyhedral Lipschitz domain (of arbitrary dimension $d\geq 2$), then all solutions to \eqref{eq:p-laplace} with boundary data $g\in W^{1}(L_p(\partial\Omega))$ are as well contained in $W^{s}(L_p(\Omega))$ for all $s<s^*$ defined by \eqref{eq:max_sobolev_reg}. 
However, he does not provide a proof of this statement.
Using this claim, the arguments in Step 1 of the proof of \autoref{theorem:Besov_reg_2d} would imply that all $p$-harmonic functions $u$ on bounded polygonal domains $\Omega$ satisfy 
\begin{equation}\label{p_harmonic_besov_reg}
	u \in B^{\sigma}_{\tau}(L_{\tau}(\Omega))
	\quad \text{for all} \quad 
	0 < \sigma 
	< \begin{cases} 
			2 					\quad & \text{if} \quad 1<p \leq 2,\\
			1 + \frac{1}{p-1}  	\quad & \text{if} \quad 2 < p < \infty\\
		\end{cases}
	\quad \text{and} \quad 
	\frac{1}{\tau} = \frac{\sigma}{2} + \frac{1}{p}.
\end{equation}
In addition, we remark that the local H\"older regularity of two-dimensional $p$-harmonic functions is known to be higher than for general solutions to the $p$-Poisson equation \link{eq:p-Poisson}: In fact, Iwaniec and Manfredi \cite{IwaMan1989} showed that in the case $d=2$ all $p$-harmonic functions are contained in $C^{\ell, \alpha}_{\loc}(\Omega)$, where $\ell \in \N$ and $0< \alpha \leq 1$ are determined by the formula
\begin{align}\label{eq:reg_parameter_p_harmonic}
	\ell+\alpha 
	= 1 + \frac{1}{6} \left( 1 + \frac{1}{p-1} + \sqrt{1 + \frac{14}{p-1} + \frac{1}{(p-1)^2}} \right). 
\end{align}
Furthermore, for $p \neq 2$ this result is known to be sharp; see \cite{IwaMan1989}. Note that for all $1 < p < \infty$ the right-hand side of \link{eq:reg_parameter_p_harmonic} indeed is larger than $1 + \alpha^*_\infty$.
In conclusion, one might expect to achieve even higher Besov regularity for $p$-harmonic functions than stated in \link{p_harmonic_besov_reg}. 
To prove this conjecture (by means of our embedding result \autoref{theorem:embedding}), we would need to exploit the sharp H\"older regularity \eqref{eq:reg_parameter_p_harmonic} instead of \autoref{Prop:hoelder_reg_plane}; provided we could show that $p$-harmonic functions belong to $C^{\ell, \alpha}_{\gamma, \loc}(\Omega)$ for these $\ell$ and $\alpha$, as well as for sufficiently small values of $\gamma$, and provided that Ebmeyer's claim holds true.
Unfortunately, sufficient estimates for the parameter $\gamma$ do not seem to exist, yet.

\section{Appendix} \label{sec:App}
This final part of the paper is concerned with estimates needed in our proofs, as well as with auxiliary assertions that are of interest on their own.

To begin with, we state the following well-known Whitney-type estimates which can be found, e.g., in DeVore \cite[Subsection 6.1]{Dev1998}.
Here and in what follows we let $\Pi_k(S)$ denote the set of all polynomials $\P$ on some bounded and simply connected set $S \subset\R^d$, $d\in\N$, which possess a total degree $\deg\P$ not larger than $k\in\N_0$.
As usual, $\lceil x \rceil$ (and $\lfloor x \rfloor$, respectively) means the smallest (largest) integer larger (smaller) or equal to $x\in\R$.

\begin{Prop}[Whitney's estimate] \label{Prop:Whitney}
	For $d\in\N$ let $Q$ denote an arbitrary cube in $\R^d$ with sides parallel to the coordinate axes.
	Moreover, 
	\begin{itemize}
		\item[(i)] let $1 \leq p,q \leq \infty$ and $k\in\N$ with $k > d\,\max\{0, 1/q - 1/p\}$. Then it holds
			\begin{align*}
				\inf_{\P \in \Pi_{k - 1}(Q)} \norm{f - \P \sep L_p(Q) } 
				\leq C \abs{Q}^{k/d + 1/p - 1/q} \abs{f}_{W^k(L_q(Q))},
			\end{align*}
			whenever the right-hand side is finite. Therein the constant $C$ depends only on $k$.
		\item[(ii)] let $1 \leq p \leq \infty$ and $0 < q \leq \infty$. 
			Furthermore, assume that $0<t< \infty$ satisfies $t \geq d\,\max\{0, 1/q - 1/p\}$. Then we have
			\begin{align*}
				\inf_{\P \in \Pi_{\lceil t \rceil - 1}(Q)} \norm{f - \P \sep L_p(Q) } 
				\leq C \abs{Q}^{t/d + 1/p - 1/q} \abs{f}_{B^t_q(L_q(Q))},
			\end{align*}
			whenever the right-hand side is finite. Here the constant $C$ depends only on $t$.
	\end{itemize}
\end{Prop}

In the proof of our general embedding result (\autoref{theorem:embedding}) the subsequent bound is used. 
As no explicit derivation of this quite natural assertion seems to be available in the literature, a detailed proof is added here for the reader's convenience.
\begin{Prop}\label{Prop:semi_norms}
	For $d\in\N$ let $Q$ denote some open cube in $\R^d$ with sides parallel to the coordinate axes. 
	Then for all $\ell \in \N_0$ and $0<\alpha\leq 1$ it holds
	\begin{equation*}
		\abs{g}_{B^{\ell+\alpha}_\infty(L_\infty(Q))}
		\lesssim \abs{g}_{C^{\ell,\alpha}(Q)},
	\end{equation*}
	whenever the right-hand side is finite.
\end{Prop}
\begin{proof}
\emph{Step 1.}
Assume that $\ell=0$.
Then, for $0< \alpha < 1$, the assertion follows from the definition of the involved semi-norms; see \link{def:hoelder_seminorm} and \link{def:Besov_semi} in \autoref{sec:spaces}. 
If $\alpha=1$, then we use the triangle inequality to see that for all $h\in\R^d$ it holds
\begin{equation}\label{est:Delta2}
	\norm{\Delta_h^{2}(g,\cdot) \sep L_\infty(Q_{2,h}) } 
	= \norm{\Delta^1_h(g,\cdot+h)-\Delta^1_h(g,\cdot) \sep L_\infty(Q_{2,h}) } 
	\lesssim \norm{\Delta_h^{1}(g,\cdot) \sep L_\infty(Q_{1,h}) },
\end{equation}
where we recall that for $r\in\N$ the set $Q_{r,h}$ denotes the collection of all $x\in Q$ such that $[x,x+rh]\subset Q$.
Then, as before, the claim directly follows from the definitions of the semi-norms.

\emph{Step 2.}
Now let $\ell\in\N$.
Given $t>0$, as well as $h\in\R^d$ with $0<\abs{h}\leq t$, and any function $f$ on some domain $\Omega\subset\R^d$, the mean value theorem ensures that for all $x\in \Omega_{1,h}$ there exists some $\xi_x \in [x,x+h]\subset \Omega$ with
\begin{equation*}
	\abs{\Delta_h^1(f,x)} 
	= \abs{h \cdot \nabla f(\xi_x)}
	\leq \abs{h} \, \abs{\nabla f(\xi_x)}
	\lesssim t \, \sum_{\abs{\nu}=1} \abs{\partial^\nu f(\xi_x)},
\end{equation*}
whenever the right-hand side is finite.
Obviously, the same is true also for $h=0$.
Thus, we conclude that for every such $f$ and all $\abs{h}\leq t$
\begin{equation}\label{est:Delta1}
	\norm{\Delta_h^1(f,\cdot) \sep L_\infty(\Omega_{1,h})}
	\lesssim t\, \sup_{x\in \Omega_{1,h}} \sum_{\abs{\nu}=1} \abs{\partial^\nu f(\xi_x)}
	\leq t \sum_{\abs{\nu}=1} \norm{ \partial^\nu f \sep L_\infty(\Omega)}.
\end{equation}
Observe that $r:=\floor{\ell+\alpha}+1\geq 2$ for all $0<\alpha\leq 1$.
Therefore, if we use \link{est:Delta1} for $f:=\Delta_h^{r-1}(g,\star)$ together with the linearity of $\partial^\nu$ and $\Delta_h^{r-1}$, 
\begin{align*}
	\abs{g}_{B^{\ell+\alpha}_\infty(L_\infty(Q))} 
	&= \sup_{t>0} t^{-(\ell+\alpha)} \sup_{h\in\R^d,\abs{h}\leq t} \norm{\Delta_h^1(\Delta_h^{r-1}(g,\star),\cdot) \sep L_\infty(Q_{r,h})} \\
	&\lesssim \sup_{t>0} t^{-(\ell+\alpha)} \sup_{h\in\R^d,\abs{h}\leq t} t \sum_{\abs{\nu}=1} \norm{ \partial^\nu \Delta_h^{r-1}(g,\star) \sep L_\infty(\Omega_{r-1,h})} \\
	&\leq \sum_{\abs{\nu}=1} \sup_{t>0} t^{-(\ell+\alpha)+1} \sup_{h\in\R^d,\abs{h}\leq t} \norm{ \Delta_h^{r-1}(\partial^\nu g,\cdot) \sep L_\infty(\Omega_{r-1,h})}.
\end{align*}
If necessary, we can iterate this argument and deduce
\begin{equation}\label{est:r1}
	\abs{g}_{B^{\ell+\alpha}_\infty(L_\infty(Q))} 
	\lesssim \sum_{\abs{\nu} = r-1} \sup_{t>0}  t^{-(\ell+\alpha)+r-1} \sup_{h\in\R^d,\abs{h}\leq t} \norm{ \Delta_h^{1}(\partial^\nu g,\cdot) \sep L_\infty(\Omega_{1,h})}.
\end{equation}
For $0<\alpha<1$ it is $r-1=\ell$. 
Consequently, in this case we obtain
\begin{align}
	\abs{g}_{B^{\ell+\alpha}_\infty(L_\infty(Q))} 
	&\lesssim \sum_{\abs{\nu} = \ell} \sup_{t>0} \sup_{h\in\R^d,\abs{h}\leq t} \frac{\norm{ \partial^\nu g(\cdot+h)-\partial^\nu g(\cdot) \sep L_\infty(\Omega_{1,h})}}{t^\alpha} \label{eq:finalsum} \\
	&= \sum_{\abs{\nu} = \ell} \sup_{\substack{x,y\in Q,\\ x \neq y}} \frac{\abs{ \partial^\nu g(x)-\partial^\nu g(y)}}{\abs{x-y}^\alpha}. \nonumber
\end{align}
Since the last term equals $\abs{g}_{C^{\ell,\alpha}(Q)}$, this shows the claim in the case $\alpha<1$.

Finally, we note that if $\alpha=1$, then $r\geq 3$. Thus, by means of the same (iterative) argument as above, this time we derive
\begin{align*}
	\abs{g}_{B^{\ell+\alpha}_\infty(L_\infty(Q))} 
	&\lesssim \sum_{\abs{\nu} = r-2} \sup_{t>0}  t^{-(\ell+\alpha)+r-2} \sup_{h\in\R^d,\abs{h}\leq t} \norm{ \Delta_h^{2}(\partial^\nu g,\cdot) \sep L_\infty(\Omega_{2,h})}
\end{align*}
instead of \link{est:r1}.
Using $r-2=\ell$ in conjunction with an estimate similar to \link{est:Delta2} from Step 1 this allows to conclude \link{eq:finalsum} also for this case.
Hence, the proof is complete.
\end{proof}

In \autoref{Rem:locally_weighted_hoelder_spaces}, among other things, we stated that intersections of locally weighted H\"older spaces (as introduced in \autoref{subsec:Hoelder}) with certain Besov spaces form Banach spaces w.r.t.\ the canonical maximum norm. \autoref{Prop:intersect} below is devoted to this claim. The subsequent three lemmata are used to derive a sound mathematical proof.
\begin{Prop}\label{Prop:intersect}
For $d\in\N$ let $\Omega\subset\R^d$ be a bounded Lipschitz domain and for $\ell\in\N_0$, $0<\alpha\leq 1$, as well as $\gamma>0$, let $C^{\ell,\alpha}_{\gamma,\loc}(\Omega)$ denote a locally weighted H\"older space.
Then for all $s>0$ and $1\leq p,q \leq \infty$ the space 
\begin{equation}\label{eq:intersect}
	B^s_q(L_p(\Omega)) \cap C^{\ell,\alpha}_{\gamma,\loc}(\Omega)
\end{equation}
endowed with the norm 
\begin{equation}\label{def:max_norm}
	\norm{\,\cdot\,} 
	= \max\!\left\{ \norm{\,\cdot \sep B^s_q(L_p(\Omega)}, \abs{\,\cdot\,}_{C^{\ell,\alpha}_{\gamma,\loc}} \right\}
\end{equation}
is a Banach space.
\end{Prop}
\begin{proof}
Since $\norm{\,\cdot \sep B^s_q(L_p(\Omega)}$ is a norm on $B^s_q(L_p(\Omega))$ and $\abs{\,\cdot\,}_{C^{\ell,\alpha}_{\gamma,\loc}}$ defines a semi-norm for $C^{\ell,\alpha}_{\gamma,\loc}(\Omega)$, it obviously holds that $\norm{\,\cdot\,}$ is a norm for the space \link{eq:intersect}.
To show completeness, let $\{f_j\}_{j \in \N_0}$ be a Cauchy sequence in \link{eq:intersect} with respect to $\norm{\,\cdot\,}$. 
Then, by completeness of the Besov space, there exists some $f \in B^s_q(L_p(\Omega))$ such that 
\begin{align}
\label{eq:Besov_conv}
	f_j \rightarrow f \quad \text{in} \quad B^s_q(L_p(\Omega)), \quad \text{as } j \rightarrow \infty.
\end{align}
This clearly remains true for all restrictions of $f_j$ and $f$, respectively, e.g., when $\Omega$ is replaced by an open ball $\mathring{B} \subset \Omega$. 

In the following, we will show that  $f_j$ converges to $f$ with respect to $\abs{\,\cdot\,}_{C^{\ell,\alpha}_{\gamma,\loc}}$, too. 
Let $B=B_r(x_0) \subset \Omega$ be a non-empty closed ball such that $B_{c\,r}(x_0)$ is still contained in $\Omega$ for some $c>1$. Given some function $g \in C^{\ell}(B)$ we denote by $T^{\ell,x_0}[g]$ its Taylor polynomial of degree $\ell$ at $x_0$, i.e.,
\begin{align*}
T^{\ell,x_0}[g](x)=\sum_{|\nu|\leq \ell} \frac{\partial^{\nu}g(x_0)}{\nu!}(x-x_0)^\nu, \quad x \in B.
\end{align*}

\emph{Step 1.} 
Here we prove that, if $\{f_j\}_{j \in \N_0}$ is a Cauchy sequence w.r.t.\ $\abs{\cdot}_{C^{\ell,\alpha}(B)}$, then 
\begin{align}\label{eq:difference_seq}
	\{f_j-T^{\ell,x_0}[f_j]\}_{j \in \N_0}
\end{align}
forms a Cauchy sequence with respect to the norm in the H\"older space $C^{\ell,\alpha}(B)$,
\begin{align*}
\norm{\cdot \sep C^{\ell,\alpha}(B)} = \norm{\cdot \sep C^\ell(B)} + \abs{\cdot}_{C^{\ell,\alpha}(B)}.
\end{align*}
Since the definition of the semi-norm $\abs{\cdot}_{C^{\ell,\alpha}(B)}$ given in \link{def:hoelder_seminorm} is based on derivatives of degree~$\ell$, we have
\begin{align}
\label{eq:Taylor_Hoelder}
\abs{f_j-T^{\ell,x_0}[f_j]}_{C^{\ell,\alpha}(B)}=\abs{f_j}_{C^{\ell,\alpha}(B)}.
\end{align}
Therefore it remains to show that \link{eq:difference_seq} is a Cauchy sequence with respect to the norm $\norm{\cdot \sep C^\ell(B)}$. 
For $j,k\in\N_0$ let $g_{j,k}=f_j-f_k$ and choose $\nu \in \N_0^{d}$ with $|\nu| \leq \ell$. 
Then, by linearity of the Taylor polynomial, for all $x\in B$ it holds
\begin{align}\label{eq:Taylor_est}
\begin{split}
	\partial^{\nu} \!\left( \left(f_j-T^{\ell,x_0}[f_j]\right)-\left(f_k-T^{\ell,x_0}[f_k]\right)\right)\!(x) 
	&= \partial^{\nu} \!\left(g_{j,k}- T^{\ell,x_0}[g_{j,k}]\right)\!(x) \\
	&=\partial^{\nu} g_{j,k} (x) - T^{\ell-|\nu|,x_0}[\partial^{\nu} g_{j,k}](x). 
\end{split}
\end{align} 
According to \autoref{le:Taylor_Hold_est} below, we thus have
\begin{align*}
	\sup_{x \in B} \abs{\partial^{\nu} g_{j,k} (x) - T^{\ell-|\nu|,x_0}[\partial^{\nu} g_{j,k}](x)} 
	\lesssim \abs{\partial^{\nu}g_{j,k}}_{C^{\ell-\abs{\nu},\alpha}(B)} \leq \abs{f_j-f_k}_{C^{\ell,\alpha}(B)}
\end{align*}
for all $|\nu|\leq \ell$. 
Together with \eqref{eq:Taylor_est} this shows that 
\begin{equation*}
	\norm{\left(f_j-T^{\ell,x_0}[f_j]\right)-\left(f_k-T^{\ell,x_0}[f_k]\right) \sep C^\ell(B)} 
	\lesssim \abs{f_j-f_k}_{C^{\ell,\alpha}(B)},
\end{equation*}
i.e., \link{eq:difference_seq} forms a Cauchy sequence w.r.t.\ $\norm{\cdot \sep C^\ell(B)}$. 
This observation in conjunction with \eqref{eq:Taylor_Hoelder} finally proves that $\{f_j-T^{\ell,x_0}[f_j]\}_{j \in \N_0}$ is a Cauchy sequence in the norm of the H\"older space $C^{\ell,\alpha}(B)$, too.

\emph{Step 2.} Of course, the space $C^{\ell,\alpha}(B)$ endowed with the norm $\norm{\cdot \sep C^{\ell,\alpha}(B)}$ is complete. 
Since we have shown that $\{f_j-T^{\ell,x_0}[f_j]\}_{j \in \N_0}$ is a Cauchy sequence with respect to this norm, there exists some $f_B \in C^{\ell,\alpha}(B)$ such that 
\begin{align*}
	\left( f_j-T^{\ell,x_0}[f_j] \right) \rightarrow f_B \quad \text{in} \quad \norm{\cdot \sep C^{\ell,\alpha}(B)}, \quad  \text{ as } j \rightarrow \infty.
\end{align*}

\emph{Step 3.} 
In the previous steps it was proven that every Cauchy sequence $\{f_j\}_{j \in \N_0}$ in $B^s_q(L_p(\Omega))\cap C^{\ell,\alpha}_{\gamma,\loc}(\Omega)$ (w.r.t.\ $\norm{\cdot}$) converges to some $f$ in $B^s_q(L_p(\Omega))$ and that for every non-empty closed ball $B=B_r(x_0) \subset \R^d$ for which $B_{c\,r}(x_0)$ is still contained in $\Omega$ the sequence $\{f_j-T^{\ell,x_0}[f_j] \}_{j \in \N_0}$ restricted to $B$ converges to some $f_B$ with respect to $\norm{\cdot \sep C^{\ell,\alpha}(B)} $. 
It remains to show that $f_j \rightarrow f$ in the semi-norm of $C^{\ell,\alpha}_{\gamma,\loc}(\Omega)$. Let $\mathring{B}$ be the interior of $B$. 
\autoref{le:uni_conv}, applied to $X=B^s_q(L_p(\mathring{B}))$, implies that the restriction of $f$ to $B$ belongs to $C^{\ell,\alpha}(B)$ and that  
\begin{align*}
f_j \rightarrow f \quad \text{with respect to} \quad \abs{\cdot}_{C^{\ell,\alpha}(B)}, \quad \text{as } j \rightarrow \infty.
\end{align*}
Since clearly, for all $j\in\N_0$ and every $B$, it holds
\begin{align*}
	\abs{f_j-f}_{C^{\ell,\alpha}(B)} \leq \lim_{k\rightarrow \infty} \abs{f_j-f_k}_{C^{\ell,\alpha}(B)},
\end{align*}
the definition of $\abs{\cdot}_{C^{\ell,\alpha}_{\gamma,\loc}(\Omega)}$ as a weighted supremum of ${C^{\ell,\alpha}(B)}$-semi-norms yields
\begin{align*}
\abs{f_j-f}_{C^{\ell,\alpha}_{\gamma,\loc}(\Omega)} \leq \lim_{k\rightarrow \infty} \abs{f_j-f_k}_{C^{\ell,\alpha}_{\gamma,\loc}(\Omega)}.
\end{align*}
Hence, from the assumption that $\{f_j\}_{j \in \N_0}$ is a Cauchy sequence in $C^{\ell,\alpha}_{\gamma,\loc}(\Omega)$ and by \eqref{eq:Besov_conv} it follows that
\begin{align*}
f_j \rightarrow f \quad \text{in} \quad B^s_q(L_p(\Omega)) \cap C^{\ell,\alpha}_{\gamma,\loc}(\Omega), \quad \text{as } j \rightarrow \infty,
\end{align*}
and thus the proof is finished.
\end{proof}

\begin{Rem}
Let $s>0$. If $0< p <1$ or $0 < q < 1$, then $B^s_q(L_p(\Omega))$ is only a quasi-Banach space, i.e., it is complete with respect to the \emph{quasi}-norm $\norm{\,\cdot \sep B^s_q(L_p(\Omega)}$. However, in the same way as in \autoref{Prop:intersect}, one can show that in this case the intersection \link{eq:intersect} endowed with the quasi-norm $\norm{\cdot}$ given by \link{def:max_norm} defines a quasi-Banach space.
\end{Rem}

\begin{lemma}\label{le:Taylor_Hold_est}
	Let $B \subset \R^d$, $d\in\N$, denote a non-trivial closed ball with center $x_0$ and let $\ell \in \N_0$. 
	For $g \in C^{\ell,\alpha}(B)$ let $T^{\ell,x_0}[g]$ be the Taylor polynomial of degree $\ell$ at $x_0$. 
	Then there exists a constant $C_{\ell,\alpha,B}>0$ such that
	\begin{align*}
		\sup_{x \in B} \abs{g (x) - T^{\ell,x_0}[g](x)} \leq C_{\ell,\alpha,B} \cdot \abs{g}_{C^{\ell,\alpha}(B)}
		\qquad \text{for all} \qquad g \in C^{\ell,\alpha}(B).
	\end{align*}
	\end{lemma}
\begin{proof}
Let $\ell\in\N$. Then, by Taylor's theorem for order $\ell-1$, for all $x \in B$ there exists a $\theta \in (0,1)$ such that
\begin{align*}
g(x)-T^{\ell,x_0}[g](x) &=g(x)-T^{\ell-1,x_0}[g](x)- \sum_{|\nu|= \ell} \frac{\partial^{\nu}g(x_0)}{\nu!}(x-x_0)^\nu \\
&=\sum_{|\nu|= \ell} \frac{\partial^{\nu}g(x_0+\theta(x-x_0))}{\nu!}(x-x_0)^\nu - \sum_{|\nu|= \ell} \frac{\partial^{\nu}g(x_0)}{\nu!}(x-x_0)^\nu \\
\end{align*}
Now, estimating the right-hand side with the help of $\abs{g}_{C^{\ell,\alpha}(B)}$ results in
\begin{align*}
	\abs{g(x)-T^{\ell,x_0}[g](x)} 
	&\leq \sum_{|\nu|= \ell} \frac{\abs{\partial^{\nu}g(x_0+\theta(x-x_0))-\partial^{\nu}g(x_0)}}{\abs{(x_0+\theta(x-x_0))-x_0}^\alpha} \, \frac{\theta^\alpha \abs{x-x_0}^{\abs{\nu}+\alpha}}{\nu!} \\
&\leq C_{\ell,\alpha,B} \abs{g}_{C^{\ell,\alpha}(B)} 
\end{align*}
for all $x \in B\setminus \{x_0\}$ and $\ell\in\N$. Since this bound obviously holds for $x=x_0$ and for $\ell=0$ as well, the claim is proven.
\end{proof}

\begin{lemma}\label{le:uni_conv}
	Let $B \subset \R^d$ be a non-trivial closed ball and denote its interior by $\mathring{B}$ . 
	Moreover, for $k,\ell \in \N_0$ with $k\leq \ell$, let $\{\P_j^{k}\}_{j\in\N_0} \subset \Pi_k(B)$ be a sequence of polynomials and suppose that $X(\mathring{B})$ denotes a quasi-Banach space of functions on $\mathring{B}$, which is continuously embedded into $\D'(\mathring{B})$. 
	Finally, assume that
	\begin{align*}
		(f_j - \P_j^{k}) \rightarrow f^1 \quad \text{with respect to} \quad \norm{\cdot \sep C^{\ell,\alpha}(B)} 
		\qquad \text{and} \qquad 
		f_j	\rightarrow f \quad \text{in} \quad  X(\mathring{B}),
	\end{align*}
	as $j$ approaches infinity.
	Then $f \in C^{\ell,\alpha}(B)$ and 
	\begin{align*}
		f_j \rightarrow f \quad \text{with respect to} \quad \abs{\cdot}_{C^{\ell,\alpha}(B)}, \quad \text{as } j \rightarrow \infty. 
	\end{align*}
\end{lemma}
\begin{proof}
Since both the spaces $C^{\ell,\alpha}(\mathring{B})$ and $X(\mathring{B})$ are continuously embedded into ${\cal D}'(\mathring{B})$, the convergence
\begin{align*}
	(f_j - \P_j^{k}) \rightarrow f^1 
	\quad \text{and} \quad 
	f_j \rightarrow f
\end{align*}
takes place in $\D'(\mathring{B})$.
Hence, $\P_j^{k} \rightarrow (f-f^1) \in {\cal D}'(\mathring{B})$, as $j\rightarrow\infty$.

On the other hand, the linear space $\Pi_k(\mathring{B})$ of polynomials of degree not larger than $k$ is closed with respect to the convergence (cf.\ \autoref{le:pol_clos} below) in ${\cal D}'(\mathring{B})$. 
Consequently, $f-f^1=:\P^k\in \Pi_k(B)$ and
\begin{align*}
	f=f^1+\P^k \in C^{\ell,\alpha}(B).
\end{align*}
Finally, as $\abs{\cdot}_{C^{\ell,\alpha}(B)}$ can not distinguish polynomials of degree less or equal to $\ell$,
\begin{align*}
	\abs{f_j-f}_{C^{\ell,\alpha}(B)}
	=\abs{\left(f_j - \P_j^{k}\right) - \left(f-\P^k\right)}_{C^{\ell,\alpha}(B)} 
	= \abs{\left(f_j - \P_j^{k}\right) - f^1}_{C^{\ell,\alpha}(B)} 
	\rightarrow 0,  \quad \text{as } j \rightarrow \infty,
\end{align*}
due to our assumption.
\end{proof}

\begin{lemma} \label{le:pol_clos}
	Let $\mathring{B}$ denote an open ball in $\R^d$, $d\in\N$.
	Then the set of polynomials $\Pi_k(\mathring{B})$ of degree at most $k\in\N_0$ on $\mathring{B}$ is closed with respect to convergence in ${\cal D}'(\mathring{B})$.
\end{lemma}
\begin{proof}
For all $\{\P_j^k\}_{j \in \N_0} \subset \Pi_k(\mathring{B})$ with
\begin{align*}
	\P_j^k \rightarrow \P \in {\cal D}'(\mathring{B}), \quad \text{as } j \rightarrow \infty,
\end{align*}
we have to show that $\P \in \Pi_k(\mathring{B})$. 
We shall prove this statement by induction on $k \in \N_0$. 
Let $k=0$. Then $\P_j^0 \equiv a_j \in \R$ is a sequence of constants converging to $\P \in {\cal D}'(\mathring{B})$, i.e.,
\begin{align*}
	a_j \int_B \varphi(x) \d x 
	= \int_B \P_j^0(x)\, \varphi(x) \d x \rightarrow \P(\varphi) 
	\quad \text{ for all } 
	\quad \varphi \in \D(B), \quad \text{as } j \rightarrow \infty.
\end{align*}
Obviously, the sequence $\{a_j\}_{j \in \N_0}$ has to be bounded in $\R$ and hence there is a subsequence $\{a_{j_{\ell}}\}_{\ell \in \N_0}$ with $a_{j_{\ell}} \rightarrow a \in \R$, as $\ell \rightarrow \infty$. 
By uniqueness of convergence of this subsequence it holds
\begin{align*}
	\P(\varphi) 
	= a \int_B \varphi(x) \d x 
\end{align*}
and thus $\P\equiv a \in \Pi_0(\mathring{B})$.

Let us now assume that $k \in \N$ and that the statement of the lemma is already shown for $0\leq \ell\leq k-1$. 
In addition, let $\nu \in \N_0^d$ with $|\nu|=k$ be a given multi-index. 
If $\P_j^k \rightarrow \P$ in ${\cal D}'(\mathring{B})$, then also $\partial^{\nu} \P_j^k \rightarrow \partial^{\nu}\P$ in ${\cal D}'(\mathring{B})$, as $j \rightarrow \infty$. 
But, for all $j\in\N_0$, $\partial^{\nu} \P_j^k \equiv a_j^{\nu} \in \R$ is a polynomial of degree $0$. Hence, by the base step of the induction, the sequence $\{\partial^{\nu} \P_j^k\}_{j\in\N_0}$ converges to some constant $a^{\nu}$ in ${\cal D}'(\mathring{B})$. This shows that 
\begin{align*}
	\P_j^{k-1}:= \P_j^k - \sum_{|\nu|=k} \frac{\partial^{\nu} \P_j^k}{\nu!}x^{\nu} 
	\quad \text{tends to} \quad
	\tilde{\P} := \P- \sum_{|\nu|=k} \frac{a^{\nu}}{\nu!}x^{\nu}
	\quad \text{in} \quad {\cal D}'(\mathring{B}), \quad \text{as } j \rightarrow \infty.
\end{align*}
Since $\P_j^{k-1}$ belongs to $\Pi_{k-1}(\mathring{B})$, by induction it follows that $\tilde{\P}\in\Pi_{k-1}(\mathring{B})$, too.
Therefore, $\P$ belongs to $\Pi_{k}(\mathring{B})$ and the proof is complete.
\end{proof}

\bibliographystyle{is-abbrv}
\bibliography{AG_Numerik}

\begin{thebibliography}{10}

\bibitem{Ada1975}
R.~A. Adams.
\newblock {\em Sobolev Spaces}.
\newblock Academic Press, New York, 1975.

\bibitem{Aro1986}
G.~Aronsson.
\newblock Construction of singular solutions to the {$p$}-harmonic equation and
  its limit equation for {$p=\infty$}.
\newblock {\em Manuscripta Math.}, 56\penalty0 (2):\penalty0 135--158, 1986.

\bibitem{BerLof1976}
J.~Bergh and J.~L{\"o}fstr{\"o}m.
\newblock {\em Interpolation {S}paces. {A}n {I}ntroduction}, volume 223 of {\em
  Grundlehren der mathematischen {W}issenschaften}.
\newblock Springer-{V}erlag, Berlin, 1976.

\bibitem{Coh2003}
A.~Cohen.
\newblock {\em Numerical {A}nalysis of {W}avelet {M}ethods}, volume~32 of {\em
  Studies in Mathematics and its Applications}.
\newblock North-Holland, Amsterdam, 2003.

\bibitem{CohDahDev2001}
A.~Cohen, W.~Dahmen, and R.~DeVore.
\newblock Adaptive wavelet methods for elliptic operator equations --
  {C}onvergence rates.
\newblock {\em Math. Comp.}, 70\penalty0 (233):\penalty0 27--75, 2001.

\bibitem{Dah1999c}
S.~Dahlke.
\newblock {B}esov regularity for elliptic boundary value problems on polygonal
  domains.
\newblock {\em Appl. Math. Lett.}, 12:\penalty0 31--38, 1999.

\bibitem{DahDahDev1997}
S.~Dahlke, W.~Dahmen, and R.~DeVore.
\newblock Nonlinear approximation and adaptive techniques for solving elliptic
  operator equations.
\newblock In W.~Dahmen, A.~Kurdila, and P.~Oswald, editors, {\em Multiscale
  Wavelet Methods for Partial Differential Equations}, pages 237--283. Academic
  Press, San Diego, 1997.

\bibitem{DahDev1997}
S.~Dahlke and R.~A. DeVore.
\newblock Besov regularity for elliptic boundary value problems.
\newblock {\em Comm. Partial Differential Equations}, 22\penalty0
  (1-2):\penalty0 1--16, 1997.

\bibitem{DahNovSic2006b}
S.~Dahlke, E.~Novak, and W.~Sickel.
\newblock Optimal approximation of elliptic problems by linear and nonlinear
  mappings {II}.
\newblock {\em J. Complexity}, 22:\penalty0 549--603, 2006.

\bibitem{DahSic2009}
S.~Dahlke and W.~Sickel.
\newblock Besov regularity for the {P}oisson equation in smooth and polyhedral
  cones.
\newblock In V.~Maz'ya, editor, {\em Sobolev Spaces in Mathematics II},
  volume~9 of {\em International Mathematical Series}, pages 123--145. Springer
  New York, 2009.

\bibitem{DahSic2013}
S.~Dahlke and W.~Sickel.
\newblock On {B}esov regularity of solutions to nonlinear elliptic partial
  differential equations.
\newblock {\em Rev. Mat. Complut.}, 26:\penalty0 115--145, 2013.

\bibitem{Dau1988}
I.~Daubechies.
\newblock Orthonormal bases of compactly supported wavelets.
\newblock {\em Commun. Pure Appl. Math.}, 41\penalty0 (7):\penalty0 909--996,
  1988.

\bibitem{Dau1992}
I.~Daubechies.
\newblock {\em {T}en {L}ectures on {W}avelets}, volume~61 of {\em CBMS--NSF
  Regional Conference Series in Applied Math.}
\newblock SIAM, Philadelphia, PA, 1992.

\bibitem{Dev1998}
R.~A. DeVore.
\newblock Nonlinear approximation.
\newblock {\em Acta Numer.}, 7:\penalty0 51--150, 1998.

\bibitem{DevPop1988}
R.~A. DeVore and V.~A. Popov.
\newblock Interpolation of {B}esov spaces.
\newblock {\em Trans. Amer. Math. Soc.}, 305\penalty0 (1):\penalty0 397--414,
  1988.

\bibitem{DevSha1984}
R.~A. DeVore and R.~C. Sharpley.
\newblock {\em Maximal {F}unctions {M}easuring {S}moothness}, volume 47(293) of
  {\em Mem. Am. Math. Soc.}
\newblock AMS, Providence, RI, 1984.

\bibitem{DevSha1993}
R.~A. DeVore and R.~C. Sharpley.
\newblock Besov spaces on domains in $\mathbb{R}^d$.
\newblock {\em Trans. Amer. Math. Soc.}, 335\penalty0 (2):\penalty0 843--864,
  1993.

\bibitem{DiB1983}
E.~DiBenedetto.
\newblock ${\uppercase{c}}^{1+\alpha}$ local regularity of weak solutions of
  degenerate elliptic equations.
\newblock {\em Nonlinear Anal. TMA}, 7:\penalty0 827--850, 1983.

\bibitem{DieKapSch2012}
L.~Diening, P.~Kaplick{\'y}, and S.~Schwarzacher.
\newblock B{MO} estimates for the {$p$}-{L}aplacian.
\newblock {\em Nonlinear Anal.}, 75\penalty0 (2):\penalty0 637--650, 2012.

\bibitem{Dis2003}
S.~Dispa.
\newblock Intrinsic characterizations of {B}esov spaces on {L}ipschitz domains.
\newblock {\em Math. Nachr.}, 260:\penalty0 21--33, 2003.

\bibitem{Dob1989}
M.~Dobrowolski.
\newblock On quasilinear elliptic equations in domains with conical boundary
  points.
\newblock {\em J. Reine Angew. Math.}, 394:\penalty0 186--195, 1989.

\bibitem{Dob2010}
M.~Dobrowolski.
\newblock {\em Angewandte {F}unktionalanalysis. {F}unktionalanalysis,
  {S}obolev-{R}\"aume und {E}lliptische {D}ifferentialgleichungen}.
\newblock Springer, Berlin, 2nd revised and extended ({G}erman) edition, 2010.

\bibitem{Ebm2002}
C.~Ebmeyer.
\newblock Mixed boundary value problems for nonlinear elliptic systems with
  p-structure in polyhedral domains.
\newblock {\em Math. Nachr.}, 236:\penalty0 91--108, 2002.

\bibitem{Eva1982}
L.~Evans.
\newblock A new proof of local ${\uppercase{c}}^{1,\alpha}$ regularity for
  solutions of certain degenerate elliptic p.d.e.
\newblock {\em J. Differential Equations}, 45:\penalty0 356--373, 1982.

\bibitem{FraJaw1990}
M.~Frazier and B.~Jawerth.
\newblock A discrete transform and decompositions of distribution spaces.
\newblock {\em J. Funct. Anal.}, 93\penalty0 (1):\penalty0 34--170, 1990.

\bibitem{Han2012}
M.~Hansen.
\newblock {$N$}-term approximation rates and {B}esov regularity for elliptic
  {PDE}s on polyhedral domains.
\newblock Technical Report 2012-41, Seminar for Applied Mathematics, ETH
  Z{\"u}rich, 2012.
\newblock To appear in J. Found. Comput. Math., 2014.

\bibitem{HanSic2011}
M.~Hansen and W.~Sickel.
\newblock Best m-term approximation and {L}izorkin-{T}riebel spaces.
\newblock {\em J. Approx. Theory}, 163\penalty0 (8):\penalty0 923--954, 2011.

\bibitem{HedNet2007}
L.~I. Hedberg and Y.~Netrusov.
\newblock {\em {A}n {A}xiomatic {A}pproach to {F}unction {S}paces, {S}pectral
  {S}ynthesis, and {L}uzin {A}pproximation}, volume 188(882) of {\em Mem. Am.
  Math. Soc.}
\newblock AMS, Providence, RI, 2007.

\bibitem{IwaMan1989}
T.~Iwaniec and J.~Manfredi.
\newblock Regularity of $p$-harmonic functions on the plane.
\newblock {\em Rev. Mat. Iberoamericana}, 5\penalty0 (1-2):\penalty0 1--19,
  1989.

\bibitem{JerKen1995}
D.~Jerison and C.~Kenig.
\newblock The inhomogeneous {D}irichlet problem in {L}ipschitz domains.
\newblock {\em J. Funct. Anal.}, 130\penalty0 (1):\penalty0 161--219, 1995.

\bibitem{Kry1996}
N.~V. Krylov.
\newblock {\em Lectures on {E}lliptic and {P}arabolic {E}quations in {H}\"older
  {S}paces}, volume~12 of {\em Graduate Studies in Mathematics}.
\newblock AMS, Providence, RI, 1996.

\bibitem{KuuMin2014}
T.~Kuusi and G.~Mingione.
\newblock Guide to nonlinear potential estimates.
\newblock {\em Bull. Math. Sci.}, 4\penalty0 (1):\penalty0 1--82, 2014.

\bibitem{Kyr1996}
G.~C. Kyriazis.
\newblock Wavelet coefficients measuring smoothness in {$H_p(\mathbb{R}^d)$}.
\newblock {\em Appl. Comput. Harmon. Anal.}, 3:\penalty0 100--119, 1996.

\bibitem{Lew1983}
J.~Lewis.
\newblock Regularity of the derivatives of solutions to certain degenerate
  elliptic equations.
\newblock {\em Indiana Univ. Math. J.}, 32\penalty0 (6):\penalty0 849--858,
  1983.

\bibitem{LinLin2013}
E.~Lindgren and P.~Lindqvist.
\newblock Regularity of the $p$-{\uppercase{p}}oisson equation in the plane.
\newblock Technical Report~13, Institut Mittag-Leffler, Royal Swedish Academy
  of Sciences, 2013/2014.
\newblock Available at arXiv:1311.6795v2.

\bibitem{Lin2006}
P.~Lindqvist.
\newblock {\em Notes on the $p$-Laplace equation}.
\newblock University of Jyv\"askyl\"a, Department of Mathematics and
  Statistics, 2006.

\bibitem{Lio1969}
J.~L. Lions.
\newblock {\em Quelques {M}\'ethodes de {R}\'esolution des {P}robl\`emes aux
  {L}imites non {L}in\'eaires}.
\newblock Etudes mathematiques. Dunod-Gauthier-Villars, 1969.
\newblock 554 pp.

\bibitem{MazRos2003}
V.~G. Mazya and J.~Ro{\ss}mann.
\newblock Weighted {$L_p$} estimates of solutions to boundary value problems
  for second order elliptic systems in polyhedral domains.
\newblock {\em Z. Angew. Math. Mech.}, 87\penalty0 (7):\penalty0 435--467,
  2003.

\bibitem{Mey1992}
Y.~Meyer.
\newblock {\em Wavelets and Operators}, volume~37 of {\em Cambridge Studies in
  Advanced Mathematics}.
\newblock Cambridge University Press, 1992.

\bibitem{Ryc1999}
V.~S. Rychkov.
\newblock On restrictions and extensions of the {B}esov and
  {T}riebel-{L}izorkin spaces with respect to {L}ipschitz domains.
\newblock {\em J. London Math. Soc. (2)}, 60\penalty0 (1):\penalty0 237--257,
  1999.

\bibitem{Sav1998}
G.~Savar{\'e}.
\newblock Regularity results for elliptic equations in {L}ipschitz domains.
\newblock {\em J. Funct. Anal.}, 152:\penalty0 176--201, 1998.

\bibitem{Schn2011}
C.~Schneider.
\newblock Traces of {B}esov and {T}riebel-{L}izorkin spaces on domains.
\newblock {\em Math. Nachr.}, 284\penalty0 (5-6):\penalty0 572--586, 2011.

\bibitem{Sci2014}
B.~Sciunzi.
\newblock Regularity and comparison principles for {$p$}-{L}aplace equations
  with vanishing source term.
\newblock {\em Commun. Contemp. Math.}, In Press, 2014.

\bibitem{Tei2014}
E.~V. Teixeira.
\newblock Regularity for quasilinear equations on degenerate singular sets.
\newblock {\em Math. Ann.}, 358\penalty0 (1-2):\penalty0 241--256, 2014.

\bibitem{Tol1984}
P.~Tolksdorf.
\newblock Regularity for a more general class of quasilinear elliptic
  equations.
\newblock {\em J. Differential Equations}, 51\penalty0 (1):\penalty0 126--150,
  1984.

\bibitem{Tri1983}
H.~Triebel.
\newblock {\em Theory of Function Spaces}.
\newblock Birkh{\"a}user, Basel/Boston/Stuttgart, 1983.

\bibitem{Tri2006}
H.~Triebel.
\newblock {\em Theory of {F}unction {S}paces {III}}.
\newblock Birkh\"auser, Basel, 2006.

\bibitem{Uhl1977}
K.~Uhlenbeck.
\newblock Regularity for a class of nonlinear elliptic systems.
\newblock {\em Acta Math.}, 138:\penalty0 219--240, 1977.

\bibitem{Ura1968}
N.~N. Ural'ceva.
\newblock Degenerate quasilinear elliptic systems.
\newblock {\em Zap. Nauchn. Sem. S.-Peterburg. Otdel. Mat. Inst. Steklov.
  (POMI)}, 7:\penalty0 184--222, 1968.

\bibitem{Woj1997}
P.~Wojtaszczyk.
\newblock {\em A Mathematical Introduction to Wavelets}.
\newblock Cambridge University Press, 1997.

\end{thebibliography}
%
\newpage
\section*{}
\noindent \textsc{Stephan~Dahlke, Christoph~Hartmann, and Markus~Weimar} \\
Philipps-University Marburg \\
Faculty of Mathematics and Computer Science, Workgroup Numerics and Optimization \\
Hans-Meerwein-Stra{\ss}e, Lahnberge \\
35032 Marburg, Germany \\
\textsl{E-mail}: \{dahlke, hartmann, weimar\}@mathematik.uni-marburg.de \\[2ex]
\textsc{Lars Diening} \\
Ludwig Maximilian University Munich \\
Institute of Mathematics \\
Theresienstra{\ss}e 39 \\
80333 Munich, Germany \\
\textsl{E-mail}: diening@mathematik.uni-muenchen.de \\[2ex]
\textsc{Benjamin Scharf} \\
TU Munich \\
Faculty of Mathematics \\
Boltzmannstra{\ss}e 3 \\
85748 Garching (Munich), Germany \\
\textsl{E-mail}: scharf@ma.tum.de \\[2ex]
\end{document}